\documentclass[a4paper]{article}
\usepackage{amsmath,amsthm,amssymb}
\usepackage[all]{xy}

\usepackage{version}
\newenvironment{NB}{
\color{red}{\bf NB}. \footnotesize
}{}

\excludeversion{NB}
\excludeversion{NB2}

\newtheorem{thm}{Theorem}[section]
\newtheorem{defn}[thm]{Definition}
\newtheorem{ex}[thm]{Example}
\newtheorem{prop}[thm]{Proposition}
\newtheorem{cor}[thm]{Corollary}
\newtheorem{lem}[thm]{Lemma}

\newtheorem{rem}[thm]{Remark}

\newcommand{\mf}[1]{{\mathfrak{#1}}}
\newcommand{\mr}[1]{{\mathrm{#1}}}
\newcommand{\mb}[1]{{\mathbf{#1}}}
\newcommand{\bb}[1]{{\mathbb{#1}}}
\newcommand{\mca}[1]{{\mathcal{#1}}}

\newcommand{\Hom}{\mr{Hom}}
\newcommand{\Ext}{\mr{Ext}}
\newcommand{\End}{\mr{End}}

\newcommand{\vEnd}{\mca{E}\mr{nd}}
\newcommand{\ext}{\mr{ext}}

\newcommand{\im}{\mr{im}}

\newcommand{\dimv}{\underline{\dim}}

\newcommand{\Z}{\bb{Z}}
\newcommand{\tZ}{\tilde{\bb{Z}}}
\newcommand{\C}{\bb{C}}
\newcommand{\CP}{\bb{P}}
\newcommand{\OO}{\mca{O}}

\newcommand{\R}{\bb{R}}
\newcommand{\LL}{\bb{L}}
\newcommand{\QQ}{\bb{Q}}

\newcommand{\Q}{\tilde{Q}}
\newcommand{\A}{\tilde{A}}
\newcommand{\V}{\tilde{V}}
\newcommand{\tz}{\tilde{\zeta}}
\newcommand{\vv}{\mb{v}}
\newcommand{\q}{\mb{q}}

\newcommand{\ys}{\mca{Y}_\sigma}

\newcommand{\pervs}{\mr{Per}(\ys/\mca{X})}

\newcommand{\coh}{\mr{Coh}(\mca{Y})}
\newcommand{\cohs}{\mr{Coh}(\ys)}

\newcommand{\amod}{A\text{-}\mr{mod}}
\newcommand{\afmod}{A\text{-}\mr{fmod}}

\newcommand{\hafmod}{\tilde{A}\text{-}\mr{fmod}}

\newcommand{\M}[2]{\mf{M}_{#1}(#2)}
\newcommand{\sM}[2]{\mf{M}^{\sigma}_{#1}(#2)}

\newcommand{\HN}{the Harder-Narasimhan filtration }
\newcommand{\JH}{a Jordan-H\"older filtration }

\newcommand{\del}{|\Delta|}

\newcommand{\h}{\frac{1}{2}}

\newcommand{\hind}{\tilde{I}}

\newcommand{\hhhind}{\left\{\frac{3}{2},\ldots,N-\frac{3}{2}\right\}}

\newcommand{\hI}{\hat{I}}
\newcommand{\tI}{\tilde{I}}

\newcommand{\bal}{\bar{\alpha}}

\newcommand{\uj}{\underline{j}}
\newcommand{\ul}{\underline{l}}

\newcommand{\rs}{\Lambda}

\title{Derived categories of small toric Calabi-Yau $3$-folds and curve counting invariants}
\author{Kentaro Nagao}
\begin{document}

\maketitle

\begin{abstract}
We first construct a derived equivalence between a small crepant resolution of an affine toric Calabi-Yau $3$-fold and a certain quiver with a superpotential.
Under this derived equivalence we establish a wall-crossing formula for the generating function of the counting invariants of perverse coherent sheaves. 
As an application we provide some equations on Donaldson-Thomas, Pandharipande-Thomas and Szendroi's invariants. 
\end{abstract}

\section*{Introduction}
This is a subsequent paper of \cite{nagao-nakajima}. 
We study {\it variants\/} of {\it Donaldson-Thomas} (DT in short) {\it invariants} on {\it small} crepant resolutions of affine toric Calabi-Yau $3$-folds. 

The original Donaldson-Thomas invariants of a Calabi-Yau $3$-fold $\mca{Y}$ are defined by virtual counting of moduli spaces of ideal sheaves $\mca{I}_Z$ of $1$-dimensional closed subschemes $Z\subset \mca{Y}$ (\cite{thomas-dt}, \cite{behrend-dt}).
These are conjecturally equivalent to {\it Gromov-Witten invariants} after normalizing the contribution of 0-dimensional sheaves (\cite{mnop}).

A variant has been introduced Pandharipande and Thomas (PT in short) as virtual counting of moduli spaces of stable {\it coherent systems} (\cite{pt1}). 
They conjectured these invariants also coincide with DT invariants after suitable normalization and mentioned that the coincidence should be recognized as a wall-crossing phenomenon.
Here, a coherent system is a pair of a coherent sheaf and a morphism to it from the structure sheaf, which is first introduced by Le Potier in his study on moduli problems (\cite{lepotier}).  
Note that an ideal sheaf $\mca{I}_Z$ is the kernel of the canonical surjections from the structure sheaf $\OO_{\mca{Y}}$ to the structure sheaf $\OO_Z$. 
So in this sense DT invariants also count coherent systems.



On the other hand, 
a variety sometimes has a derived equivalence with a noncommutative algebra.
A typical example is a {\it noncommutative crepant resolution} of a Calabi-Yau $3$-fold introduced by Michel Van den Bergh (\cite{vandenbergh-3d}, \cite{vandenbergh-nccr}). 
In the case of \cite{vandenbergh-3d}, the Abelian category of modules of the noncommutative crepant resolution corresponds to the Abelian category of {\it perverse coherent sheaves} in the sense of Tom Bridgeland (\cite{bridgeland-flop}).
Recently, Balazs Szendroi proposed to study counting invariants of ideals of such noncommutative algebras (\cite{szendroi-ncdt}). He called these invariants {\it noncommutative Donaldson-Thomas} (NCDT in short) {\it invariants}.
He originally studied on the conifold, but his definition works in more general settings (\cite{young-mckay}, \cite{ncdt-brane}).

Inspired by his work, Hiraku Nakajima and the author introduced {\it perverse coherent systems} (pairs of a perverse coherent sheaf and a morphism to it from the structure sheaf) and study their moduli spaces and counting invariants (\cite{nagao-nakajima}). 
This attempt seems successful since 
\begin{itemize}
\item we can describe explicitly a space of stability parameters with a chamber structure, and 
\item at certain chambers, the moduli spaces in DT, PT and NCDT theory are recovered.
\end{itemize}
Moreover, in the conifold case, we established the wall-crossing formula for the generating functions of counting invariants of perverse coherent systems and provide some equations on DT, PT and NCDT invariants.  
The chamber structure and the wall-crossing formula formally look very
similar to the counter parts for moduli spaces of perverse coherent
sheaves on the blow-up of a complex surface studied earlier by
Nakajima and Yoshioka \cite{ny-perv1,ny-perv2}.

The purpose of this paper is to show the wall-crossing formula (Theorem \ref{thm3.16.}) for general small crepant resolutions of toric Calabi-Yau $3$-folds. 
Here we say a crepant resolutions of affine toric Calabi-Yau $3$-fold is {\it small} when the dimensions of the fibers are less than $2$. 
In such cases, the lattice polygon in $\R^2$ corresponding to the affine toric Calabi-Yau $3$-fold does not have any lattice points in its interior.
Such lattice polygons are classified up to equivalence into the following two cases:
\begin{itemize}
\item trapezoids with heights $1$, or
\item the right isosceles triangle with length $2$ isosceles edges. 
\end{itemize}
In this paper we study the first case. 
The arguments in \S \ref{derived-equiv} and \S \ref{counting} work for the second case as well\footnote{Theorem \ref{thm_mr} does not hold in the second case. We can not apply \cite[Theorem 7.1]{ncdt-brane} since mutation of a quiver associated to a dimer model is not associated to any dimer model in general.}. 



In \S \ref{derived-equiv}, we construct derived equivalences between small crepant resolutions of affine toric Calabi-Yau $3$-folds and certain quivers with superpotentials\footnote{Moreover, the module category corresponds to the category of perverse sheaves under this derived equivalence. This is stronger than what we can get from the general results such as \cite{ncdt-brane, Davison, Broomhead, Ishii_Ueda2}.}.
In \S \ref{tilting-generator}, using toric geometry, we construct tilting vector bundles given by Van den Bergh (\cite{vandenbergh-3d})  explicitly.
Then, we review Ishii and Ueda's construction of crepant resolutions as moduli spaces of representations of certain quivers with superpotentials (\cite{ishii-ueda}) in \S \ref{as-moduli}.
In \S \ref{specific-parameter} we show the tautological vector bundles on the moduli spaces coincide with the tilting bundles given in \S \ref{tilting-generator}. Using such moduli theoretic description, we calculate the endomorphism algebras of the tilting bundles in \S \ref{end-alg}. 
 
The argument in \S \ref{counting} is basically parallel to \cite{nagao-nakajima}.
In our case, the fiber on the origin of the affine toric variety is the type $A$ configuration of $(-1,-1)$- or $(0,-2)$-curves. 
A wall in the space of stability parameters is a hypersurface which is perpendicular to a root vector of the root system of type $\hat{A}$.    
Stability parameters in the chambers adjacent to the wall corresponding to the imaginary root realize DT theory and PT theory (\cite[\S 2]{nagao-nakajima}).
Note that the story is completely parallel to that of type $\hat{A}$ quiver varieties (of rank $1$), which are the moduli spaces of framed representations of type $\hat{A}$ preprojective algebras (\cite{quiver1}, \cite{quiver2}, \cite{quiver3}). 
Quiver varieties associated with a stability parameter in a chamber adjacent to the imaginary wall realize Hilbert schemes of points on the minimal resolution of the Kleinian singularities of type $A$, whose exceptional fiber is the type $A$ configuration of $(-2)$-curves (\cite{nakajima-lec-note}, \cite{kuznetsov-stability}). 

Our main result is the wall-crossing formula for the generating functions of the Euler characteristics of the moduli spaces (Theorem \ref{thm3.16.}). 
The contribution of a wall depends on the information of self-extensions of stable objects on the wall. Note that in the conifold case (\cite{nagao-nakajima}) every wall has a single stable object on it and every stable object has a trivial self-extension. Computations of self-extensions are done in \S \ref{appendix}. 

\begin{NB}
Since the sets of torus fixed points on the moduli spaces in DT, PT and NCDT theory are isolated, and we can show that DT, PT and NCDT invariants coincide up to signs with the Euler characteristics of the moduli spaces. 
In particular, the wall-crossing formula provides a product expansion formula of the generating functions of PT invariants. 
The indices in this formula are nothing but the BPS state counts $n_{g,\beta}$ (\cite{gopakumar-vafa}, \cite{hosono-saito-takahashi}, \cite{toda-gv}) in the sense of Pandharipande-Thomas (\cite[\S 3]{pt1})\footnote{An algorithm to extract Gopakumar-Vafa invariants of the toric Calabi-Yau $3$-folds from the topological vertex expression is given in \cite{strip}. The author was informed on this reference by Yukiko Konishi.}. 

\end{NB}

The DT, PT and NCDT invariants are defined as weighted Euler characteristics of the moduli spaces weighted by Behrend functions.
It is the purpose of \S \ref{mutation} to compare the weighted Euler characteristics and the Euler characteristics of the moduli spaces for a generic stability parameter.
First, we provide alternative descriptions of the moduli spaces.
Given a quiver with a superpotential  $A=(Q,\omega)$, we can {\it mutate} it at a vertex $k$ to provide a new quiver with a superpotential  $\mu_k(A)=(\mu_k(Q),\mu_k(\omega))$.
For a generic stability parameter $\zeta$, we can associate a sequence $k_1,\ldots k_r$ of vertices and the moduli space of $\zeta$-stable $A$-modules is isomorphic to the moduli space of finite dimensional quotients of a module over the quiver with the superpotential $\mu_{k_r}\circ\cdots\circ\mu_{k_1}(A)$.
As a consequence, we can apply Behrend-Fantechi's torus localization theorem (\cite{behrend-fantechi}) to show that
the weighted Euler characteristics coincide with the Euler characteristics up to signs.

As in \cite{nagao-nakajima}, our formula does not cover the wall corresponding to the DT-PT conjecture. 
We can provide the wall-crossing formula for this wall applying Joyce's formula (\cite{joyce-4})\footnote{Kontsevich-Soibelman's (partly conjectural) formula (\cite{ks}) also covers the setting in this paper. }\footnote{After this paper submitted \begin{itemize}\item Euler characteristic version of the DT-PT conjecture was proved by Toda, Thomas-Stoppa, Bridgeland (\cite{toda-dtpt, thomas_stoppa, bridgeland-dtpt}) using Joyce's arguments, and \item Joyce-Song provided an extension of the wall-crossing formula in \cite{joyce-4} to weighted Euler characteristics and an application for non-commutative Donaldson-Thomas invariants (\cite{joyce-song}).\end{itemize}}.


\subsection*{Acknowledgement}

The author is grateful to Hiraku Nakajima for collaborating in the paper \cite{nagao-nakajima} and for many valuable discussion. 

He thanks Kazushi Ueda for patiently teaching his work on brane tilings, 
Tom Bridgeland for explaining his work, 
Yan Soibelman for helpful comments.

He also thanks Yoshiyuki Kimura and Michael Wemyss for useful discussions, Akira Ishii, Yukari Ito, Osamu Iyama, Yukiko Konishi, Sergey Mozgovoy and Yukinobu Toda for helpful comments. 
The author is supported by JSPS Fellowships for Young Scientists No.19-2672.

\section{Derived equivalences}\label{derived-equiv}
\subsection{tilting generators}\label{tilting-generator}
Let $N_0>0$ and $N_1\geq 0$ be integers such that $N_0\geq N_1$ and set $N=N_0+N_1$.
We set
\begin{align*}
I&=\left\{1,\ldots,N-1\right\},\\
\hat{I}&=\left\{0,1,\ldots,N-1\right\},\\
\tilde{I}&=\left\{\h,\frac{3}{2},\ldots,N-\h\right\},\\
\tilde{\Z}&=\left\{n+\h\,\Big|\, n\in\Z\right\}.
\end{align*}
For $l\in\Z$ and $j\in\tZ$, let $\underline{l}\in\hI$ and $\underline{j}\in\tI$ be the elements 
such that $l-\underline{l}\equiv j-\underline{j}\equiv 0$ modulo $N$. 

We denote by $\Gamma$ the quadrilateral (or the triangle in the cases $N_1=0$) in $(\R^2)_{x,y}=\{(x,y)\}$ with vertices $(0,0)$, $(0,1)$, $(N_0,0)$ and $(N_1,1)$.
Let $M^*\simeq \Z^3$ be the lattice with basis $\{x^*,y^*,z^*\}$ and
$M:=\mr{Hom}_{\Z}(M^*,\Z)$ be the dual lattice of $M^*$. 
Let $\Delta$ denote the cone of the image of $\Gamma$ under the inclusion
\[
(\R^2)_{x,y}\hookrightarrow \{(x,y,z)\}=M^*_\R:=M^*\otimes \R\colon (x,y)\mapsto (x,y,1)
\]
and consider the semigroup
\[
S_\Delta=\Delta^\vee\cap M:=\{u\in M\mid \langle u,v\rangle\geq 0\  (\forall v\in\Delta)\}.
\]
Let $R=R_\Gamma:=\C[S_\Delta]$ be the semi-group algebra and $\mca{X}=\mca{X}_\Gamma:=\mr{Spec}(R_\Gamma)$ the $3$-dimensional affine toric Calabi-Yau variety corresponding to $\Delta$. 

Let $\{x,y,z\}\subset M$ be the dual basis. 
The semigroup is generated by 
\begin{align}
X&:=x,\label{eq1}\\
Y&:=-x-(N_0-N_1)y+N_0z,\label{eq2}\\
Z&:=y,\\
W&:=-y+z,
\end{align}
and they have a unique relation $X+Y=N_1Z+N_0W$.
So we have
\[
R\simeq \C[X,Y,Z,W]/(XY-Z^{N_1}W^{N_0}).
\]

A partition $\sigma$ of $\Gamma$ is a pair of functions $\sigma_x\colon\tI\to\tZ$ and $\sigma_y\colon\tI\to\{0,1\}$ such that 
\begin{itemize}
\item $\sigma(i):=(\sigma_x(i),\sigma_y(i))$ gives a bijection between $\tI$ and the following set:
\[
\left\{\left(\h,0\right),\left(\frac{3}{2},0\right),\ldots,\left(N_0-\h,0\right),\left(\h,1\right),\left(\frac{3}{2},1\right),\ldots,\left(N_1-\h,1\right)\right\},
\]  
\item if $i<j$ and $\sigma_y(i)=\sigma_y(j)$ then $\sigma_x(i)>\sigma_x(j)$.
\end{itemize}
Giving a partition $\sigma$ of $\Gamma$ is equivalent dividing $\Gamma$ into $N$-tuples of triangles $\{T_i\}_{i\in\tI}$ with area $1/2$ so that $T_i$ has $(\sigma_x(i)\pm 1/2, \sigma_y(i))$ as its vertices.
Let $\Gamma_\sigma$ be the corresponding diagram, $\Delta_\sigma$ be the fan and $f_\sigma\colon \mca{Y}_\sigma\to \mca{X}$ be the crepant resolution of $\mca{X}$. 

We denote by $D_{\varepsilon,x}$ ($\varepsilon=0,1$ and $0\leq x\leq N_\varepsilon$) the divisor of $\mca{Y}_\sigma$ corresponding to the lattice point $(x,\varepsilon)$ in the diagram $\Gamma_\sigma$. 
Note that any torus equivariant divisor is described as a linear combination of $D_{\varepsilon,x}$'s. For a torus equivariant divisor $D$ let $D(\varepsilon,x)$ denote its coefficient of $D_{\varepsilon,x}$.
The support function $\psi_D$ of $D$ is the piecewise linear function on $|\Delta_\sigma|$ such that $\psi_D((x,\varepsilon,1))=-D(\varepsilon,x)$ and such that $\psi_D$ is linear on each cone of $\Delta_\sigma$. 
We sometimes denote the restriction of $\psi_D$ on the plane $\{z=1\}$ by $\psi_D$ as well. 

\begin{defn}\label{def1}
For $i\in\tI$ and $k\in I$ we define effective divisors $E^\pm_i$ and $F^\pm_k$ by
\begin{align*}
&E^+_i=\sum_{j=\sigma_x(i)+\h}^{N_{\sigma_y(i)}}D_{\sigma_y(i),j},
&F^+_k=\sum_{i=\h}^{k-\h}E^+_i,\\
&E^-_i=\sum_{j=0}^{\sigma_x(i)-\h}D_{\sigma_y(i),j},
&F^-_k=\sum_{i=k+\h}^{N-\h}E^-_i.
\end{align*}
\end{defn}

\begin{ex}\label{example}
Let us consider as an example the case $N_0=4$, $N_1=2$ and 
\[
(\sigma(i))_{i\in \tI}=\left(\left(\frac{7}{2},0\right),\left(\frac{3}{2},1\right),\left(\frac{5}{2},0\right),\left(\frac{3}{2},0\right),\left(\frac{1}{2},1\right), \left(\frac{1}{2},0\right)\right).
\]
We show the corresponding diagram $\Gamma_\sigma$ in Figure \ref{fig:Q}.
The divisors are given as follows:
\[
\begin{array}{cc}
E_\h^+:=\left[\begin{array}{cccccc}
0&0&0& & \\
0&0&0&0&1
\end{array}\right],&
F_1^+:=\left[\begin{array}{cccccc}
0&0&0& & \\
0&0&0&0&1
\end{array}\right],\vspace{5pt}\\
E_\frac{3}{2}^+:=\left[\begin{array}{cccccc}
0&0&1& & \\
0&0&0&0&0
\end{array}\right],&
F_2^+:=\left[\begin{array}{cccccc}
0&0&1& & \\
0&0&0&0&1
\end{array}\right],\vspace{5pt}\\
E_\frac{5}{2}^+:=\left[\begin{array}{cccccc}
0&0&0& & \\
0&0&0&1&1
\end{array}\right],&
F_3^+:=\left[\begin{array}{cccccc}
0&0&1& & \\
0&0&0&1&2
\end{array}\right],\vspace{5pt}\\
E_\frac{7}{2}^+:=\left[\begin{array}{cccccc}
0&0&0& & \\
0&0&1&1&1
\end{array}\right],&
F_4^+:=\left[\begin{array}{cccccc}
0&0&1& & \\
0&0&1&2&3
\end{array}\right],\vspace{5pt}\\
E_\frac{9}{2}^+:=\left[\begin{array}{cccccc}
0&1&1& & \\
0&0&0&0&0
\end{array}\right],&
F_5^+:=\left[\begin{array}{cccccc}
0&1&2& & \\
0&0&1&2&3
\end{array}\right],\vspace{5pt}\\
E_\frac{11}{2}^+:=\left[\begin{array}{cccccc}
0&0&0& & \\
0&1&1&1&1
\end{array}\right],&
F_6^+:=\left[\begin{array}{cccccc}
0&1&2& & \\
0&1&2&3&4
\end{array}\right].\vspace{5pt}\\
\end{array}
\]
Here the $(\varepsilon,x)$-th matrix element represent the coefficient of the divisor $D_{\varepsilon,x}$. 
\end{ex}
\begin{figure}[htbp]
  \centering
  \input{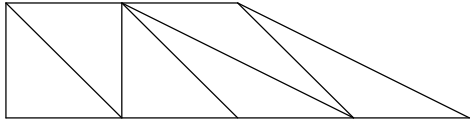}
  \caption{$\Gamma_\sigma$}
  \label{fig:Q}
\end{figure}

\begin{lem}\label{lem2}
\begin{enumerate}
\item[(1)] $\OO_{\ys}(E^+_i+E^-_i)\simeq \OO_{\ys}$,
\item[(2)] $\OO_{\ys}(F^+_k)\simeq \OO_{\ys}(F^-_k)$.
\end{enumerate}
\end{lem}
\begin{proof}
We have
\[
\psi_{E^+_i+E^-_i}=
\begin{cases}
y-z & (\sigma_y(i)=0),\\
-y & (\sigma_y(i)=1),
\end{cases}
\]
so the equation (1) follows. 
Now, for the equation (2) it is enough to show that the divisor
\[
F^+_N:=\sum_{i=\h}^{N-\h}E^+_i
\]
gives the trivial bundle.  
In fact, we have $\psi_{F^+_N}=-x$.
\end{proof}
We denote the line bundle $\OO_{\mca{Y}_\sigma}(F^+_k)\simeq \OO_{\mca{Y}_\sigma}(F^-_k)$ on $\mca{Y}_\sigma$ by $L_k$.
We set 
\[
F^\pm=\sum_{k=1}^{N-1}F^\pm_k,\quad L=\bigotimes_{k=1}^{N-1}L_k.
\]
\begin{ex}
In the case Example \ref{example},
\[
F^+=\left[\begin{array}{cccccc}
0&1&5& & \\
0&0&2&5&10
\end{array}\right].
\]
\end{ex}

\begin{lem}\label{localconvex}
For $i\in \tI$ we have
\begin{align*}
&F^+\left(\sigma_y(i),\sigma_x(i)+\h\right)-F^+\left(\sigma_y(i),\sigma_x(i)-\h\right)\\
&\ =F^+\left(\sigma_y(i+1),\sigma_x(i+1)+\h\right)-F^+\left(\sigma_y(i+1),\sigma_x(i+1)-\h\right)+1.
\end{align*}
\end{lem}
\begin{proof}
First, note that 
\[
E^+_j\left(\sigma_y(i),\sigma_x(i)+\h\right)-E^+_j\left(\sigma_y(i),\sigma_x(i)-\h\right)=\delta_{i,j}.
\]
So we have
\[
F^+_k\left(\sigma_y(i),\sigma_x(i)+\h\right)-F^+_k\left(\sigma_y(i),\sigma_x(i)-\h\right)=\begin{cases}
0 & (k<i),\\
1 & (k>i),
\end{cases}
\]
and so 
\[
F^+\left(\sigma_y(i),\sigma_x(i)+\h\right)-F^+\left(\sigma_y(i),\sigma_x(i)-\h\right)=N-i-\h.
\]
Thus the claim follows.
\end{proof}

\begin{prop}
The line bundles $L$ is generated by its global sections. 
\end{prop}
\begin{proof}
It is enough to prove that the support function $\psi_{F^+}$ is upper convex (\cite[\S 3.4]{fulton-toric}).
It is enough to prove that $\psi_{F^+}$ is upper convex on the interior of $T_{k-\h}\cup T_{k+\h}$ for any $k\in I$.
We denote 
the edge which is the intersection of $T_{k-\h}$ and $T_{k+\h}$ by $l_{k}$. 
The configurations of $l_k$, $T_{k-\h}$ and $T_{k+\h}$ are classified into the following two cases:
\begin{enumerate}
\item[(1)] The union of $T_{k-\h}$ and $T_{k+\h}$ is a parallelogram and $l_k$ is its diagonal. In this case, the point $(\sigma_x(k+\h)+\h,\sigma_y(k+\h))$ is the intersection of $l_k$ and $l_{k-1}$, the point $(\sigma_x(k-\h)+\h,\sigma_y(k-\h))$ is the other end of $l_{k-1}$.
\item[(2)] The union of $T_{k-\h}$ and $T_{k+\h}$ is a triangle and $l_k$ is its median line. In this case, the point $(\sigma_x(k+\h)+\h,\sigma_y(k+\h))$ is the middle point and $\sigma_x(k-\h)=\sigma_x(k+\h)+1$, $\sigma_y(k-\h)=\sigma_y(k+\h)$.
\end{enumerate}
In both cases it follows from Lemma \ref{localconvex} that $\psi_{F^+}$ is upper convex on $T_{k-\h}\cap T_{k+\h}$. 
\end{proof}

Given a divisor $D$ the space of global sections of the line bundle $\OO_{\ys}(D)$ is described as follows:
\[
H^0(\ys,\OO_{\ys}(D))\simeq \bigoplus_{u\in S_\Delta^0(D)}\C\cdot e_u,
\]
where
\[
S_\Delta^0(D):=\{u\in M\mid \langle u,v\rangle\geq \psi_D(v)\ (\forall v\in \del)\}.
\]
For $u\in M$ we define
\[
Z_D(u):=\{v\in|\Delta|\mid \langle u,v\rangle\geq \psi_D(v)\}.
\]
Then the cohomology of the line bundle $\OO_{\mca{Y}_\sigma}(D)$ is given as follows (\cite[\S 3.5]{fulton-toric}):
\[
H^k(\mca{Y}_\sigma,\OO_{\mca{Y}_\sigma}(D))\simeq\bigoplus_{u\in M}H^k(\del,\del\backslash Z_D(u);\C)
\]
(here the notation $H^k$ represents the sheaf cohomology in the left hand side but represents the relative singular cohomology in the right hand side).
We have the exact sequence of relative cohomologies 
\[
\begin{array}{cccccccc}
0 & \to & H^0(\del,\del\backslash Z_D(u)) & \to & H^0(\del) & \overset{i_*}{\longrightarrow} & H^0(\del\backslash Z_D(u))  \\
 & \to & H^1(\del,\del\backslash Z_D(u)) & \to & H^1(\del) & \longrightarrow & \cdots 
\end{array}
\]
Note that $H^0(\del)=\C$, $H^1(\del)=0$ and if $\del\backslash Z_D(u)$ is not empty then $i_*$ does not vanish. 
We define
\[
Z^\circ_D(u):=\{v\in|\Delta|\cap \{z=1\}\mid \langle u,v\rangle< \psi_D(v)\},
\]
then $\del\backslash Z_D(u)$ is homeomorphic to $Z^\circ_D(u)\times \R$. 

Now, in our situation it follows from the convexity of $\psi_{-F^+}$ that the number of connected components of $Z^\circ_{-F^+}(u)$ is at most $2$. 
Let us denote
\[
S_\Delta^1(-F^+):=\{u\in M\mid \text{$Z^\circ_{-F^+}(u)$ has two connected components}\}.
\]
Then we have 
\begin{equation}\label{eq-h1}
H^1(\mca{Y}_\sigma,\OO_{\mca{Y}_\sigma}(-F^+))\simeq \bigoplus_{u\in S_\Delta^1(-F^+)}\C\cdot f_u
\end{equation}
and the $R$-module structure is given by 
\[
e_{u'}\cdot f_{u}=
\begin{cases}
\,f_{u+u'}& (u+u'\in S_\Delta^1(-F^+)),\\
\,0 &(u+u'\notin S_\Delta^1(-F^+))
\end{cases}
\]
for $u\in S_\Delta^1(-F^+)$ and $u'\in S_\Delta$.

For $i\in\tI$, let $t^i_{F^+}\in M$ the element such that $\langle t^i_{F^+},*\rangle\equiv \psi_{-F^+}$ on the triangle $T_i$. 
Note that $t^i_{F^+}\in S_\Delta^1(-F^+)$ for $i\in\hhhind$.
\begin{prop}
The set $\left\{f_{t^i_{F^+}}\,\big|\, i\in\hhhind\right\}$ is a set of generators of $H^1(\mca{Y}_\sigma,\OO_{\mca{Y}_\sigma}(-F^+))$ as an $R$-module.
\end{prop}
\begin{proof}
It is enough to check that for any $u\in S^1_\Delta(-F^+)$ there exist $i\in\hhhind$ and $u'\in S_\Delta$ and such that $u=t^i_{F^+}+u'$. 
For $\varepsilon\in\{0,1\}$ we put 
\[
m_\varepsilon:=\max_{0\leq j\leq N_\varepsilon}\left(\langle u,(j,\varepsilon,1)\rangle-F^+(\varepsilon,j)\right)\geq 0.
\]
Let $u'\in S_\Delta$ be the element such that
$\langle u',(x,\varepsilon,1)\rangle=m_\varepsilon$ for any $x$ and $\varepsilon$. 
Note that $Z^\circ_{-F^+}(u)$ has two connected components, $Z^\circ_{-F^+}(u)\subset Z^\circ_{-F^+}(u-u')$ and there exist $x_\varepsilon$'s such that $\langle u-u',(x_\varepsilon,\varepsilon,1)\rangle-F^+(\varepsilon,j)=0$. Thus $Z^\circ_{-F^+}(u-u')$ also has two connected components, that is, $u-u'\in S_\Delta^1(-F^+)$.
The function $\psi_{-F^+}-\langle u-u',*\rangle$ is upper convex and does not take negative values on $\Gamma$. 
So if for some $x_\varepsilon$'s we have $F^+(\varepsilon,j)-\langle u-u',(x_\varepsilon,\varepsilon,1)\rangle=0$, then $(x_0,0)$ and $(x_1,1)$ should be the end points of some edge in $\Gamma_\sigma$. 
Since $u-u'\in S_\Delta^1(-F^+)$, $(x_0,x_1)$ can be neither $(0,0)$ nor $(N_0,N_1)$. 
So $(x_0,0)$ and $(x_1,1)$ is the end points of an edge $l_k$ for some $k\in I$. 
By Lemma \ref{localconvex}, $\psi_{-F^+}$ coincides with $\langle u-u',*\rangle$ on either $T_k$ or $T_{k+1}$ since otherwise $\psi_{-F^+}-\langle u-u',*\rangle$ takes negative values on $\Gamma$. 
Thus the claim follows.
\end{proof}

For a divisor $D$ and a effective divisor $E$, let $1_{D,E}$ be the canonical inclusion $\OO_{\mca{Y}_\sigma}(D)\hookrightarrow\OO_{\mca{Y}_\sigma}(D+E)$.

Let us denote effective divisors
\[
G_i^+=\sum_{k=1}^{i-\h}F^+_k,\quad G_i^-=\sum_{k=i+\h}^{N-1}F^-_k.
\]
Note that $G_i^++G_i^-$ is linearly equivalent to $F^+$ by lemma \ref{lem2} and 
\[
\left(\psi_{G_i^+}+\psi_{G_i^-}\right)\Big|_{T_i} = 0.
\]
Hence we have
\[
\psi_{-F^+}-\langle t^i_{F^+},*\rangle = \psi_{-G_i^+}+\psi_{-G_i^-}.
\]
Let us consider the following sequence:
\[
0\to \OO_{\mca{Y}_\sigma}
\longrightarrow \OO_{\mca{Y}_\sigma}(G_i^+)\oplus \OO_{\mca{Y}_\sigma}(G_i^-)
\longrightarrow \OO_{\mca{Y}_\sigma}(G_i^++G_i^-)\to 0.
\]
Here the first map is given by $1_{0,G_i^+}\oplus (-1_{0,G_i^-})$ and the second map is given by $\left(1_{G_i^+,G_i^-}\right)+\left(1_{G_i^-,G_i^+}\right)$.
\begin{prop}\label{smallseq}
The above sequence is exact and corresponds to the element $f_{t_{F^+}^i}\in H^1({\mca{Y}_\sigma},\OO_{\mca{Y}_\sigma}(-F^+))$. 
\end{prop}
\begin{proof}
Let $U_j\simeq \C^3=\mr{Spec}(\C[x_j,y_j.z_j])$ be the affine chart corresponding to the triangle $T_j$. 
It is enough to show that the sequence is exact on each $U_j$. 
Let $P_a$ ($a=x,y,z$) be the vertices of $T_j$. Then we have
\begin{itemize}
\item for $j<i$, $\psi_{G_i^+}(P_a)=\psi_{G_i^++G_i^-}(P_a)(=:d_{j,a})$ and $\psi_{G_i^-}(P_a)=0$,
\item for $j=i$, $\psi_{G_i^+}(P_a)=\psi_{G_i^++G_i^-}(P_a)=\psi_{G_i^-}(P_a)=0$,
\item for $j<i$, $\psi_{G_i^-}(P_a)=\psi_{G_i^++G_i^-}(P_a)(=:d_{j,a})$ and $\psi_{G_i^+}(P_a)=0$.
\end{itemize}
The sequence in the claim is restricted on $U_j$ to the following sequence of $\C[x_j,y_j.z_j]$-modules:
\[
0\to (0,0,0)\overset{1\oplus (-1)}{\longrightarrow} (0,0,0)\oplus (d_{j,x},d_{j,y},d_{j,z})\overset{1+1}{\longrightarrow}(d_{j,x},d_{j,y},d_{j,z})\to0.
\]
Here $(0,0,0)$ (resp. $(d_{j,x},d_{j,y},d_{j,z})$) is spanned by 
\[
\{x^ay^bz^c\mid a,b,c\geq 0\}\quad (\text{resp.}\ \{x^ay^bz^c\mid a\geq -d_{j,x},b\geq -d_{j,y},c\geq -d_{j,z} \}
)
\]
as a vector space and $1$ is the map which maps $x^ay^bz^c$ to $x^ay^bz^c$.
We can verify this is exact.
The corresponding element in $H^1(\mca{Y}_\sigma,\OO_{\mca{Y}_\sigma}(-F^+))$ can be checked by the \v{C}ech argument in \cite[\S 3.5]{fulton-toric}.
\end{proof}
For $k\in I$ and $i\in\left\{\frac{3}{2},\ldots,N-\frac{3}{2}\right\}$ we define divisors
\[
F^\Delta_k:=F^+_k-F^-_k,\quad H_i:=\sum_{k=1}^{i-\h}F^\Delta_k,\quad I_k:=H_{k-\h}+F^+_{k},
\]
and the exact sequence
\[
0\to\bigoplus_{i\in\hhhind}\OO_{\mca{Y}_\sigma}(H_i)\to \bigoplus_{k\in I}\OO_{\mca{Y}_\sigma}(I_k)\to \OO_{\mca{Y}_\sigma}(F^+)\to 0. 
\]
The first map the sum of compositions of the maps
\[
\left(1_{H_i,F^-_{i-\h}}\oplus-1_{H_i,F^+_{i+\h}}\right)\colon\OO_{\mca{Y}_\sigma}(H_i)\longrightarrow\OO_{\mca{Y}_\sigma}(I_{i-\h})\oplus\OO_{\mca{Y}_\sigma}(I_{i+\h})
\]
and the canonical inclusions
\[
\OO_{\mca{Y}_\sigma}(I_{i-\h})\oplus\OO_{\mca{Y}_\sigma}(I_{i+\h})\hookrightarrow\bigoplus_{k\in I}\OO_{\mca{Y}_\sigma}(I_k).
\]
The second map is the sum of $1_{I_k,F^+-I_k}$'s. 
\begin{prop}\label{univseq}
The above sequence gives the universal extension corresponding to the set $\{f_{t^i_{F^+}}\mid i\in\hhhind\}$ of generators of $H^1({\mca{Y}_\sigma},\OO_{\mca{Y}_\sigma}(-F^+))$ as $H^0({\mca{Y}_\sigma},\OO_{\mca{Y}_\sigma})$-module.  
\end{prop}
\begin{proof}
For $i\in\hhhind$ we can check the sequence
\[
0\to \OO_{\mca{Y}_\sigma}(H_i)\to \left(\oplus_{k}\OO_{\mca{Y}_\sigma}(I_k)\right)\big/\left(\oplus_{j\neq i}\OO_{\mca{Y}_\sigma}(H_j)\right)\to \OO_{\mca{Y}_\sigma}(F^+)\to 0 
\]
is isomorphic to the exact sequence in Proposition \ref{smallseq}.
\end{proof}

Now we have the following theorem:
\begin{thm}
The direct sum $\OO_{\mca{Y}_\sigma}\oplus\bigoplus_{k\in I}L_k$ is a projective generator of ${}^{-1}\pervs$.
\end{thm}
\begin{proof}
This claim follows from \cite[Proposition 3.2.5]{vandenbergh-3d} and Proposition \ref{univseq}.
\end{proof}

\subsection{Crepant resolutions as moduli spaces}\label{as_moduli}
We will associate $\sigma$ with a quiver with superpotential $A_{\sigma}=(Q_{\sigma},\omega_{\sigma})$.
The set of vertices of the quiver $Q_{\sigma}$ is $\hI$, which is identified with $\Z/N\Z$. 
The set of edges of the quiver $Q_{\sigma}$ is given by
\[
H:=\left(\coprod_{i\in\hind}h^+_i\right)\sqcup\left(\coprod_{i\in\hind}h^-_i\right)\sqcup\left(\coprod_{k\in \hI _r}r_k\right).
\]
Here $h^+_i$ (resp. $h^-_i$) is an edge from $i-\h$ to $i+\h$ (resp. from $i+\h$ to $i-\h$), $r_k$ is an edge from $k$ to itself and
\[
\hI _r:=\left\{k\in \hI\,\Big|\,\sigma_y(k-\h)=\sigma_y(k+\h)\right\}.
\]
The relation is given as follows:
\begin{itemize}
\item $h^+_i\circ r_{i-\h}=r_{i+\h}\circ h^+_i$ and $r_{i-\h}\circ h^-_i=h^-_i\circ r_{i+\h}$ for $i\in \hind$ such that $i-\h$, $i+\h\in \hI _r$.
\item $h^+_i\circ r_{i-\h}=h^-_{i+1}\circ h^+_{i+1}\circ h^+_i$ and $r_{i-\h}\circ h^-_i=h^-_i\circ h^-_{i+1}\circ h^+_{i+1}$ for $i\in \hind$ such that $i-\h\in \hI _r$, $i+\h\notin \hI _r$.
\item $h^+_i\circ h^+_{i-1}\circ h^-_{i-1}=r_{i+\h}\circ h^+_i$ and $h^+_{i-1}\circ h^-_{i-1}\circ h^-_i=h^-_i\circ r_{i+\h}$ for $i\in \hind$ such that $i-\h\notin \hI _r$, $i+\h\in \hI _r$.
\item $h^+_i\circ h^+_{i-1}\circ h^-_{i-1}=h^-_{i+1}\circ h^+_{i+1}\circ h^+_i$ and $h^+_{i-1}\circ h^-_{i-1}\circ h^-_i=h^-_i\circ h^-_{i+1}\circ h^+_{i+1}$ for $i\in \hind$ such that $i-\h$, $i+\h\notin \hI _r$.
\item $h^+_{i-\h}\circ h^-_{i-\h}=h^-_{i+\h}\circ h^+_{i+\h}$ for $k\in \hI _r$.
\end{itemize}

This quiver is derived from the following bipartite graph on a $2$-dimensional torus.
Let $S$ be the union of infinite number of rhombi with edge length $1$ as in Figure \ref{fig:S} which is located so that the centers of the rhombi are on a line parallel to the $x$-axis in $\R^2$  and $H$ be the union of infinite number of hexagons with edge length $1$ as in Figure \ref{fig:H} which is located so that the centers of the hexagons are in a line parallel to the $x$-axis in $\R^2$.
\begin{figure}[htbp]
  \centering
  \input{4gon-2.tpc}
  \caption{S}
  \label{fig:S}
\end{figure}
\begin{figure}[htbp]
  \centering
  \input{6gon.tpc}
  \caption{H}
  \label{fig:H}
\end{figure}
We make the sequence $\Z\to \{S,H\}$ which maps $l$ to $S$ (resp. $H$) if $l$ module $N$ is not in $\hI _r$ (resp. is in $\hI _r$) and cover the whole plane $\R^2$ by arranging $S$'s and $H$'s according to this sequence (see Figure \ref{fig:P}). 
We regard this as a graph on the $2$-dimensional torus $\R^2/\Lambda$, where $\Lambda$ is the lattice generated by $(\sqrt{3},0)$ and $(N_0-N_1,(N_0-N_1)\sqrt{3}+N_1)$.
\begin{figure}[htbp]
  \centering
  \input{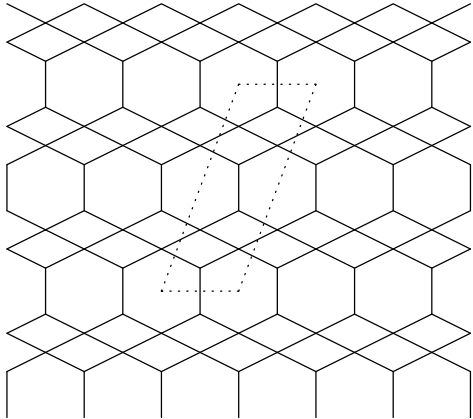}
  \caption{$P_\sigma$ in case Example \ref{example}}
  \label{fig:P}
\end{figure}
We can colored the vertices of this graph by black or white so that each edge connect a black vertex and a white one. 
Let $P_{\sigma}$ denote this bipartite graph on the torus. 
For each edge $h^\vee$ in $P_{\sigma}$, we make its dual edge $h$ directed so that we see the black end of $h^\vee$ on our right hand side when we cross $h^\vee$ along $h$ in the given direction. 
The resulting quiver coincides with $Q_{\sigma}$. 
For each vertex $q$ of $P_{\sigma}$, let $\omega_q$ be the superpotential\footnote{A superpotential of a quiver $Q$ is an element in $\C Q/[\C Q,\C Q]$, i.e. a linear combination of equivalent classes of cyclic paths in $Q$ where two paths are equivalent if they coincide after a cyclic rotation.} which is the composition of all arrows in $Q_{\sigma}$ corresponding to edges in $P_{\sigma}$ with $q$ as their ends.
We define 
\[
\omega_\sigma:=\sum_{\text{$q$ : black}}\omega_q-\sum_{\text{$q$ : white}}\omega_q.
\]
\begin{rem}\label{lem_consistency}
Take a polygon in the plane labeled by $k\in \hat{I}$.
Then, we have the bijection between the set of equivalent classes of paths starting from $k$ in $A_\sigma$ and the direct product of the set of polygons in the plane and $\Z_{\geq 0}$. 
See \cite[\S 4]{ncdt-brane}.
\end{rem}

Put $\delta=(1,\ldots,1)=\Z^{\hI}$ and take a stability parameter $\theta\in \Hom(\Z^{\hI},\R)\simeq \R^{\hI}$ so that $\theta(\delta)=0$. 
\begin{defn}
An $A_\sigma$-module $V$ is $\theta$-(semi)stable if for any nonzero submodule $0\neq S\subsetneq V$ we have $\theta(\dimv S)\,(\leq)\,0$.
\end{defn}
By \cite{king} such a stability condition coincides with a stability condition in geometric invariant theory and we can define the moduli space $\sM{\theta}{\delta}$ of $\theta$-semistable $A_\sigma$-modules $V$ such that $\dimv\,V=\delta$.
A stability parameter $\theta$ is said to be generic if $\theta$-semistability and $\theta$-stability are equivalent to each other.
\begin{thm}[\protect{\cite[Theorem 6.4]{ishii-ueda}}]
For a generic stability parameter $\theta$, the moduli space $\sM{\theta}{\delta}$ is a crepant resolution of $\mca{X}$.
\end{thm}
\begin{proof}
We can verify easily that the assumptions in \cite[Theorem6.4]{ishii-ueda} are satisfied. 
We can also check that the convex hull of height charges (see \cite[\S 2]{ishii-ueda}) is equivalent to $\Delta$ (see the description of divisors before Theorem \ref{generator} for example). 
\end{proof}
Let $T\subset \sM{\theta}{\delta}$ be the open subset consisting of representations $t$ such that $t(h^\pm_{i})\neq 0$ for any $i\in \tI$ and $t(r_k)\neq 0$ for any $k\in \hI _r$. Then $T$ is the $3$-dimensional torus acting on $\sM{\theta}{\delta}$ by edge-wise multiplications (see \cite[\S 3 and Proposition 5.1]{ishii-ueda}).
We define the map $f\colon T\to \C^*$ by $f(t)=t(e_r)\circ \cdots \circ t(e_1)$ where $e_r\cdot\cdots\cdot e_1$ is a representative of the superpotential $w_q$ for a vertex $q$ in $P_\sigma$ ($f(t)$ does not depend on the choice of the vertex $q$).
We put $T':=f^{-1}(1)\subset T$. 
This is a $3$-dimensional subtorus of $T$.


\subsection{Description at a specific parameter}\label{specific-parameter}
Let $\theta_0$ be a stability parameter so that $\theta_0(\delta)=0$ and $(\theta_0)_k<0$ for any $k\neq 0$.  

For $i\in\hind$ let $p_i\in\sM{\theta_0}{\delta}$ be the representation such that 
\[
p_i(h^+_{j})=\begin{cases}1 & (j<i),\\0 & (j\geq i),\end{cases}\quad
p_i(h^-_{j})=\begin{cases}1 & (j>i),\\0 & (j\leq i),\end{cases}\quad
p_i(r_k)=0.
\]
This is fixed by the torus action. 
We can take a coordinate $(x_i,y_i,z_i)$ on the neighborhood $U_i$ of $p_i$ in $\sM{\theta_0}{\delta}$ such that the representation $v[x_i,y_i,z_i]$ with the coordinate $(x_i,y_i,z_i)$ is give by  
\begin{align*}
v[x_i,y_i,z_i](h^+_{j})&=1\quad (j< i),\\
v[x_i,y_i,z_i](h^-_{j})&=1\quad (j> i),\\
v[x_i,y_i,z_i](h^+_{i})&=x,\\
v[x_i,y_i,z_i](h^-_{i})&=y,\\
v[x_i,y_i,z_i](q^-_{i-\h})&=v[x_i,y_i,z_i](q^+_{i+\h})=z,
\end{align*}
where
\[
q^\pm_{l}=\begin{cases}r_{l} & (l\in \hI _r),\\h^\mp_{l\pm\h}\circ h^\pm_{l\pm\h} & (l\notin \hI _r).\end{cases}
\]

For $k\in I$ let $c^1_k\in\sM{\theta_0}{\delta}$ be the representation such that 
\begin{align*}
c^1_k(h^+_{j})&=\begin{cases}1 & (j<k),\\0 & (j>k),\end{cases}\\
c^1_k(h^-_{j})&=\begin{cases}1 & (j>k),\\0 & (j<k),\end{cases}\\
c^1_k(r_l)&=0,
\end{align*}
and $\mca{C}_k$ be the closure of the orbit of $c^1_k$ with respect to the torus action.
This is isomorphic to $\CP^1$, contained in $U_{k-\h}\cap U_{k+\h}$ and 
\[
\mca{C}_k|_{U_{k-\h}}=\{y_{k-\h}=z_{k-\h}=0\},\quad \mca{C}_k|_{U_{k+\h}}=\{x_{k+\h}=z_{k+\h}=0\}.
\]
Note that the coordinate transformation is given by
\[
(x_{k+\h},y_{k+\h},z_{k+\h})\longmapsto \begin{cases}(x^2_{k-\h}y_{k-\h},x_{k-\h}^{-1},z_{k-\h}) & (k\in \hI _r)\\(x_{k-\h}z_{k-\h},x_{k-\h}^{-1},x_{k-\h}y_{k-\h}) & (k\notin \hI _r),\end{cases}
\]
Hence $\mca{C}_k$ is a $(0,-2)$-curve if $k\in \hI _r$ and a $(-1,-1)$-curve if $k\notin \hI _r$.

For a crepant resolution of $\mca{X}$ the configuration of $(0,-2)$-curves and $(-1,-1)$-curves determines its toric diagrams. So we have the following proposition:

\begin{prop}
The crepant resolution $\sM{\theta_0}{\delta}$ is isomorphic to $\mca{Y}_\sigma$. 
For $i\in\tI$ the triangle $T_i$ corresponds to the fixed point $p_i$ and for $k\in I$ the edge $l_k$ corresponds to the curve $\mca{C}_k$. 
\end{prop}

For $\varepsilon=0$ or $1$ and $0\leq x\leq N_\varepsilon$ let $d^1_{\varepsilon,x}\in\sM{\theta_0}{\delta}$ be the representation such that
\begin{align*}
d^1_{\varepsilon,x}(h^+_{j})&=\begin{cases}1 & (\sigma_y(j)\neq \varepsilon\ \text{or}\ \sigma_x(i)>x),\\0 & (\text{otherwise}),\end{cases}\\
d^1_{\varepsilon,x}(h^-_{j})&=\begin{cases}1 & (\sigma_y(j)\neq \varepsilon\ \text{or}\ \sigma_x(i)<x),\\0 & (\text{otherwise}),\end{cases}\\
d^1_{\varepsilon,x}(r_k)&=\begin{cases}1 & (\sigma_y(k\pm\h)=\varepsilon),\\0 & (\text{otherwise}),\end{cases}.
\end{align*}
Then the closure of the orbit of $d^1_{\varepsilon,x}$ with respect to the torus action coincides with the divisor $D_{\varepsilon,x}$.

\begin{figure}[htbp]\label{fig:univ}
  \centering
  \input{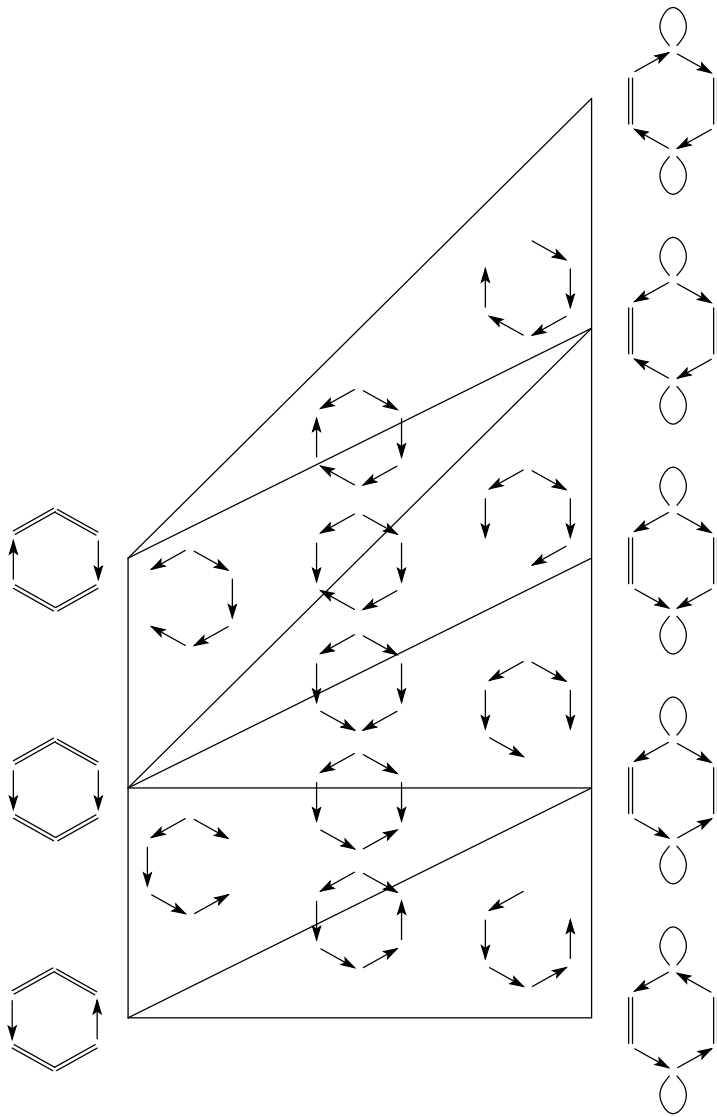}
  \caption{Universal representations on $\sM{\theta_0}{\delta}$ in case Example \ref{example}}
\end{figure}

Let $\oplus_{k\in\hI}L_k^{\sigma,\theta_0}$ be the tautological vector bundle on $\sM{\theta_0}{\delta}$, where $L_k^{\sigma,\theta_0}$ is the tautological line bundle corresponding to the vertex $k$ of the quiver $Q_\sigma$.
By tensoring a line bundle if necessary, we may assume that $L_0^{\sigma,\theta_0}\simeq \OO_{\sM{\theta_0}{\delta}}$. 
We have the tautological section of the line bundle $\Hom(L_{i\mp\h}^{\sigma,\theta_0},L_{i\pm\h}^{\sigma,\theta_0})$ corresponding to the edge $h^\pm_i$. From the description above, its divisor coincides with $E^\pm_i$ defined in \S \ref{def1}.
Hence we have 
\[
L_k^{\sigma,\theta_0}\simeq\OO(\sum_{i=\h}^{k-\h}E^+_i)\simeq L_k,
\]
where $L_k$ is defined just after \ref{lem2}.
In summary, we have the following theorem:
\begin{thm}\label{generator}
The tautological vector bundle $L^{\sigma,\theta_0}=\bigoplus_{k\in\hI}L_k^{\sigma,\theta_0}$ is a projective generator of $\pervs$.
\end{thm}
In particular, we have the following equivalence:
\[
\begin{array}{ccc}
D^b(\cohs) & \simeq & D^b(\mr{mod}(\End_{\ys}(L^{\sigma,\theta_0})))\\
\cup & & \cup\\
\pervs & \simeq & \mr{mod}(\End_{\ys}(L^{\sigma,\theta_0})).
\end{array}
\]

\subsection{Computation of the endomorphism algebra}\label{end-alg}
In this subsection, we denote $L_k^{\sigma,\theta_0}$ and $L^{\sigma,\theta_0}$ simply by $L_k^{\sigma}$ and $L^{\sigma}$ respectively.  
Since $L^\sigma$ is the tautological bundle on the moduli space of $A_\sigma$-modules, we have the tautological map 
\[
\phi\colon A_\sigma\to \End_{\ys}(L^\sigma).
\]

\begin{prop}
Let $e_k\in A_\sigma$ be the idempotent corresponding to the vertex $k$. Then the restriction of the tautological map
\[
\phi_k\colon e_kA_\sigma e_k\to \End_{\ys}(L^\sigma_k)\simeq H^0(\ys,\OO_{\ys})\simeq \C[X,Y,Z,W]/(XY-Z^{N_1}W^{N_0}).
\]
is bijective.
\end{prop}
\begin{proof}
We have the tautological section of the line bundle $\vEnd_{\ys}(L^\sigma_k)\simeq \OO_{\ys}$ corresponding to the path
\[
X_k:=h^+_{k-\h}\circ \cdots \circ h^+_{\h}\circ h^+_{N-\h}\circ \cdots \circ h^+_{k+\frac{3}{2}}\circ h^+_{k+\h}.
\]
Its divisor is 
\[
F^+_N=\sum_{i=\h}^{N-\h}E^+_i
\]
and as we have seen in Lemma \ref{lem2} we have $-\psi_{F^+_N}=x$.
So $\phi_k(X_k)$ coincides with $X$ up to scalar multiplication (see the defining equation (\ref{eq1}) of $X$ in \S \ref{tilting-generator}). 

Similarly, we put
\begin{align*}
Y_k&:=h^-_{k+\h}\circ \cdots \circ h^-_{N-\h}\circ h^-_{\h}\circ \cdots \circ h^-_{k-\frac{3}{2}}\circ h^-_{k-\h},\\
Z_k&:=
\begin{cases}
h^-_{k+\h}\circ h^+_{k+\h} & (\sigma_y(k+\h)=1),\\
h^+_{k-\h}\circ h^-_{k-\h} & (\sigma_y(k-\h)=1),\\
r_{k} & (\text{otherwise}),
\end{cases}\\
W_k&:=
\begin{cases}
h^-_{k+\h}\circ h^+_{k+\h} & (\sigma_y(k+\h)=0),\\
h^+_{k-\h}\circ h^-_{k-\h} & (\sigma_y(k-\h)=0),\\
r_{k} & (\text{otherwise}).
\end{cases}
\end{align*}
Then $\phi_k(Y_k)$, $\phi_k(Z_k)$ and $\phi_k(W_k)$ respectively coincide with $Y$, $Z$ and $W$ up to scalar multiplications. 
Hence $\phi_k$ is surjective. 
To show the injectivity, it is enough to check that $X_k$, $Y_k$, $Z_k$ and $W_k$
\begin{itemize}
\item generate $e_k Ae_k$, 
\item commute with each other, and
\item satisfy the relation $X_kY_k=Z_k^{N_0}W_k^{N_1}$.
\end{itemize}
These follow from the lemma below. 
\end{proof}
We define an equivalent relation $\sim$ on the set of all paths in $Q_\sigma$ as follows: $P_1\circ P_2\circ P_3\sim P_1\circ P_3$ if $P_2$ is one of $r_k$, $h^+_i\circ h^-_i$ or $h^-_i\circ h^+_i$.
For a path $P$, let $P_{\mr{red}}$ be the unique minimal length path which is equivalent to $P$.
We set
\begin{align*}
r_k(P)&=(\text{the number of $r_k$ appearing in $P$}),\\
h_i(P)&=(\text{the number of $h^+_i$ appearing in $P$})-(\text{the number of $h^+_i$ appearing in $P_{\mr{red}}$})\\
&=(\text{the number of $h^-_i$ appearing in $P$})-(\text{the number of $h^-_i$ appearing in $P_{\mr{red}}$}),
\end{align*}
and
\begin{align*}
z(P)&=\sum_{\begin{subarray}{c}l\in \hI _r,\\\sigma_y(l\pm\h)=1\end{subarray}}r_k(P)+\sum_{\begin{subarray}{c}i\in \hind,\\\sigma_y(i)=0\end{subarray}}h_i(P)
,\\
w(P)&=\sum_{\begin{subarray}{c}l\in \hI _r,\\\sigma_y(l\pm\h)=0\end{subarray}}r_l(P)+\sum_{\begin{subarray}{c}i\in \hind,\\\sigma_y(i)=1\end{subarray}}h_i(P).
\end{align*}
\begin{lem}
\[
P=P_{\mr{red}}\circ (Z_k)^{z(P)}\circ (W_k)^{w(P)}.
\]
\end{lem}
\begin{proof}
From the relations of $A_\sigma$ we can verify directly that for any path $P$ from $k$ to $k'$ we have
\begin{equation}\label{eq3}
Z_{k'}\circ P=P\circ Z_{k},\quad W_{k'}\circ P=P\circ W_{k}
\end{equation}
in $A_\sigma$. 

For a path $P$ in $Q_\sigma$ from $k$ we have a expression 
\[
P=P'\circ Z_{k'}\circ P''\ (\text{or}\ P'\circ W_{k'}\circ P'')
\]
for some paths $P'$, $P''$ in $Q_\sigma$ and $k'\in I$ unless $P=P_{\mr{red}}$.
This case we have 
\[
P=P'\circ P''\circ Z_{k}\ (\text{or}\ P\circ P''\circ W_{k})
\]
in $A_\sigma$ by the equation (\ref{eq3}). Then apply the same procedure for $P'\circ P''$ unless it is reduced. By the induction with respect to the lengths of paths, we can verify the claim.
\end{proof}

For $k\neq k'\in I$ let $X_{k,k'}$ (resp. $Y_{k,k'}$) be the minimal length path from $k$ to $k'$ which is a composition of $h^+_i$'s (resp. $h^-_i$'s). 
Then we have the following basis of $e_{k'}A_\sigma e_{k}$:
\[
\{(X_{k'})^n\circ X_{k,k'}\circ (Z_k)^{m}\circ (W_k)^{l}\}_{n,m,l\geq 0}\sqcup 
\{(Y_{k'})^n\circ Y_{k,k'}\circ (Z_k)^{m}\circ (W_k)^{l}\}_{n,m,l\geq 0}.
\]
We have
\begin{align*}
X_{k',k}\circ (X_{k'})^n\circ X_{k,k'}\circ (Z_k)^{m}\circ (W_k)^{l}&=(X_k)^{n+1}\circ (Z_k)^{m}\circ (W_k)^{l},\\
X_{k,k'}\circ (Y_{k'})^n\circ Y_{k,k'}\circ (Z_k)^{m}\circ (W_k)^{l}&=(Y_k) \circ (Z_k)^{m+m'}\circ (W_k)^{l+l'},
\end{align*}
where $m'$ and $l'$ is nonnegative integer such that $X_{k,k'}\circ Y_{k,k'}=(Z_k)^{m'}\circ (W_k)^{l'}$. 
In particular, we have
\begin{lem}\label{inj}
The map 
\[
X_{k,k'}\circ -\colon e_{k'}A_\sigma e_{k}\to e_{k}A_\sigma e_{k}
\]
is injective. 
\end{lem}



\begin{prop}
For $k\neq k'\in I$ the restriction of the tautological map
\[
\phi_{k,k'}\colon e_{k'}Ae_k\to \Hom_{\ys}(L_k^\sigma,L_{k'}^\sigma)
\]
is bijective.
\end{prop}
\begin{proof}
The injectivity follows from Lemma \ref{inj} and the injectivity of $\phi_k$. 
For the surjectivity, it is enough to check that $\Hom_{\ys}(L_k,L_{k'})$ is generated by the image of 
$X_{k,k'}$ and $Y_{k,k'}$ as an $R$-module.
Let $F^+_{k.k'}$ (resp. $F^-_{k.k'}$) be the sum of $E^+_i$'s (resp. $E^-_i$'s) corresponding to $X_{k,k'}$ (resp. $Y_{k,k'}$).
Note that $\psi_{F^+_{k.k'}-F^-_{k.k'}}$ is the unique element in $M$ such that $\psi_{F^+_{k.k'}-F^-_{k.k'}}(N_\varepsilon,\varepsilon,1)=\psi_{F^+_{k.k'}}(N_\varepsilon,\varepsilon,1)$ and 
such that $\psi_{F^+_{k.k'}-F^-_{k.k'}}(1,0,0)=-1$.  

Let $u\in S_\Delta^0(F^+_{k.k'})$ be an element such that $\langle u,v \rangle<0$ for some $v\in Q$. 
It is enough to check there exists $u'\in S_\Delta$ such that $u=\psi_{F^+_{k.k'}-F^-_{k.k'}}+u'$.
Since $\langle u,(0,0,1) \rangle, \langle u,(0,1,1) \rangle\geq 0$ and $\langle u,v \rangle<0$ for some $v\in Q$, we have $\langle u,(1,0,0) \rangle<0$. 
Let $u'\in S_\Delta$ be the element such that 
\begin{align*}
\langle u',(1,0,0) \rangle&=\langle u,(1,0,0)\rangle +1\\
\langle u',(N_\varepsilon,\varepsilon,1) \rangle&=\langle u,(N_\varepsilon,\varepsilon,1)\rangle+F^+_{k,k'}(\varepsilon,N_\varepsilon).
\end{align*}
It follows from the characterization of $\psi_{F^+_{k.k'}-F^-_{k.k'}}$ above that $u-u'=\psi_{F^+_{k.k'}-F^-_{k.k'}}$.
\end{proof}

In summary, we have the following theorem:
\begin{thm}
The tautological homomorphism 
\[
A_\sigma\to \End_{\ys}(L^\sigma).
\]
is an isomorphism.
\end{thm}

\section{Counting invariants}\label{counting}
From now on, we denote $Q_\sigma$ and $A_\sigma$ simply by $Q$ and $A$.
Let $Q_0$ and $Q_1$ denote the sets of vertices and edges of the quiver $Q$ respectively, and $\afmod$ denote the category of finite dimensional $A$-modules. 
\subsection{Koszul resolution}
We set a grading on the path algebra $\C Q$ such that 
\[
\mr{deg}(e_k)=0,\quad \mr{deg}(h^\pm_i)=1,\quad \mr{deg}(r_k)=2.
\]
The superpotential $\omega$ is homogeneous of degree $4$ with respect to this grading. So the quiver with superpotential $A=(Q,\omega)$ is graded $3$-dimensional Calabi-Yau algebra in the sense of \cite{graded3CY}. 

We denote the subalgebra $\C Q_0=\oplus_{k\in Q_0}\C e_k$ of $A$ by $S$. 
For an $S$-module $T$ we define an $A$-bimodule $A_T$ by
\[
A_T=A\otimes_S T\otimes_S A.
\]
For $k, k'\in I$ let $T_{k,k'}$ denote the $1$-dimensional $S$-module given by
\[
e_l\cdot 1=\delta_{k,l},\quad 1\cdot e_l=\delta_{k',l},
\]
and we set 
\[
A_k:=A_{T_{k,k}},\quad A_{k,k'}:=A_{T_{k,k'}}.
\]
Note that an element of $A_{k,k'}$ is described as a linear combination of $\{p\otimes 1\otimes q\}_{p\in Ae_k,\ q\in e_{k'}A}$.

For a quiver with superpotential $A$, the {\it Koszul complex} of $A$ is the following complex of $A$-bimodules: 
\[
0 \to \bigoplus_{k\in Q_0}A_k \overset{d_3}{\longrightarrow} \bigoplus_{a\in Q_1}A_{\mr{out}(a),\mr{in}(a)} \overset{d_2}{\longrightarrow} \bigoplus_{b\in Q_1}A_{\mr{in}(b),\mr{out}(b)} \overset{d_1}{\longrightarrow} \bigoplus_{k\in Q_0}A_k \overset{m}{\longrightarrow} A \to 0.
\]
Here $\mr{in}(a)$ (resp. $\mr{out}(a)$) is the vertex at which an arrow $a$ ends (resp. starts).
The maps $m$, $d_1$, $d_3$ are given by 
\begin{align*}
m(p\otimes 1\otimes q)&=pq\quad (p\in Ae_k,\ q\in e_kA),\\
d_1(p\otimes 1\otimes q)&=(pb\otimes 1\otimes q)-(p\otimes 1\otimes bq)
\quad (p\in Ae_{\mr{in}(b)},\ q\in e_{\mr{out}(b)}A),\\
d_3(p\otimes 1\otimes q)&=\left(\bigoplus_{a\colon \mr{in}(a)=k}pa\otimes 1\otimes q\right)-\left(\bigoplus_{a\colon \mr{out}(a)=k}p\otimes 1\otimes aq\right)\quad(p\in Ae_k,\ q\in e_kA).
\end{align*}
The map $d_2$ is defined as follows: 
Let $c$ be a cycle in the quiver $Q$. 
We define the map $\partial_{c;a,b}\colon A_{\mr{out}(a),\mr{in}(a)}\to A_{\mr{in}(b),\mr{out}(b)}$ by 
\begin{equation}\label{eq_del_c}
\partial_{c;a,b}(p\otimes 1\otimes q)=\sum_{
\begin{subarray}{c}
arbs=c
\end{subarray}
}ps\otimes 1\otimes rq.
\end{equation}
Then $d_2=\partial_{\omega}$ is defined as the linear combination of $\partial_c$'s. 

Since $A$ is graded $3$-dimensional Calabi-Yau algebra, the Koszul complex is exact (\cite[Theorem 4.3]{graded3CY}). 

Let $E$ be a finite dimensional $A$-module. 
By taking the tensor product of $E$ and the Koszul complex, we get a projective resolution of $E$:
\begin{align}
0 \to & \bigoplus_{k\in Q_0}Ae_k\otimes_\C E_k \overset{d_3}{\longrightarrow} \bigoplus_{a\in Q_1}Ae_{\mr{out}(a)}\otimes_\C E_{\mr{in}(a)} \overset{d_2}{\longrightarrow}\notag\\
&\bigoplus_{b\in Q_1}Ae_{\mr{in}(b)}\otimes_\C E_{\mr{out}(b)} \overset{d_1}{\longrightarrow} \bigoplus_{k\in Q_0}Ae_k\otimes_\C E_k \overset{m}{\longrightarrow} E \to 0.\label{eq_proj_resol}
\end{align}

\subsection{new quiver}
We make a new quiver $\Q$ by adding one vertex $\infty$ and one arrow from the vertex $\infty$ to the vertex $0$ to the original quiver $Q$. 
The original superpotential $\omega$ gives the superpotential on the new quiver $\Q$ as well. 
We set $\A:=(\Q,\omega)$ and denote the category of finite dimensional $\A$-modules by $\hafmod$.
Giving a finite dimensional $\A$-module $\V$ is equivalent to giving a pair $(V,W,i)$ of a finite dimensional $A$-module $V$, a finite dimensional vector space $W$ at the vertex $\infty$ and a linear map $i\colon W\to V_0$. 
Let $\iota\colon \hafmod\to \afmod$ be the forgetting functor mapping $(V,W,i)$ to $V$. 

We also consider the following Koszul type complex $\A$-bimodules as well:
\[
0 \to \bigoplus_{k\in Q_0}\tilde{A}_k \overset{\tilde{d}_3}{\longrightarrow} \bigoplus_{a\in Q_1}\tilde{A}_{\mr{out}(a),\mr{in}(a)} \overset{\tilde{d}_2}{\longrightarrow} \bigoplus_{b\in \Q_1}\tilde{A}_{\mr{in}(b),\mr{out}(b)} \overset{\tilde{d}_1}{\longrightarrow} \bigoplus_{k\in \Q_0}\tilde{A}_k \overset{\tilde{m}}{\longrightarrow} \A \to 0,
\]
where $\tilde{A}_i$, $\tilde{A}_{i,i'}$, $\tilde{d}_\bullet$ and $\tilde{m}$ are defined in the same way.
This is also exact. 
The exactness at the last three terms is equivalent to the definition of generators and relations of the algebra $\A$.
The exactness at the first two terms is derived from that of the exactness of the Koszul complex of $A$. 
For a finite dimensional $\A$-module, we get a projective resolution in the same way as \eqref{eq_proj_resol}.



\begin{prop}\label{prop3.1.}
For $E$, $F\in \hafmod$ we have
\begin{align*}
&\hom_{\A}(E,F)-\ext^1_{\A}(E,F)+\ext^1_{\A}(F,E)-\hom_{\A}(F,E)\\
&=\dim E_\infty\cdot \dim F_0-\dim E_0\cdot\dim F_\infty.
\end{align*}
\end{prop}
\begin{proof}
The bimodule resolution above provides a projective resolution of $E$.
We can compute $\Ext^\bullet_{\A}(E,F)$ by the complex given by applying $\Hom(-,F)$ for projective resolution of $E$:
\begin{align}
0 \to & \bigoplus_{k\in Q_0}\Hom(E_k,F_k) 
\to
\bigoplus_{a\in Q_1}\Hom(E_{\mr{in}(a)},F_{\mr{out}(a)})
\to\notag\\
&\bigoplus_{b\in Q_1}\Hom(E_{\mr{out}(b)},F_{\mr{in}(b)}) 
\to
\bigoplus_{k\in Q_0}\Hom(E_k,F_k) \to 0.
\end{align}
Let $d_i^{\A}(E,F)$ denote the derivations in the complex above.
We can compute $d_i^{A}(\iota(E),\iota(F))$ in the same way.
Let $d_i^{A}(\iota(E),\iota(F))$ denote the corresponding derivation.

Note that 
\begin{align*}
\mr{rank}\left(d_2^{\A}(E,F)\right)&=\mr{rank}\left(d_2^{A}(\iota(E),\iota(F))\right)\\
&=\mr{rank}\left(d_2^{A}(\iota(F),\iota(E))\right)\\
&=\mr{rank}\left(d_2^{\A}(F,E)\right),
\end{align*}
where the second equation comes from the self-duality of the Koszul complex of $A$.
Hence we have
\begin{align*}
&\hom_{\A}(E,F)-\ext^1_{\A}(E,F)+\ext^1_{\A}(F,E)-\hom_{\A}(F,E)\\
=\,& \sum_{h\in\Q_1}(\dim E_{\mr{out}(h)}\cdot \dim F_{\mr{in}(h)}-\dim E_{\mr{in}(h)}\cdot \dim F_{\mr{out}(h)})\\
=\,&\dim E_\infty\cdot \dim F_0-\dim E_0\cdot\dim F_\infty.
\end{align*}
Here the last equation follows from the fact that for any $i,\, j\in I$ 
\[
\sharp (\text{arrows from $i$ to $j$})=\sharp (\text{arrows from $j$ to $i$}).
\]
\end{proof}
\begin{rem}
By the Koszul resolution of $A$, the last equation in the proof is equivalent to the vanishing of the Euler form on $\afmod$. This is equivalent to the vanishing of the Euler form on the category $\mr{Coh}_{\mr{cpt}}(\mca{Y})$ of coherent sheaves on $\mca{Y}$ with compact supports. 
The vanishing on $\mr{Coh}_{\mr{cpt}}(\mca{Y})$ follows from Hirzebruch-Riemann-Roch theorem. 
\end{rem}
\begin{rem}
Joyce-Song proved more general statement (\cite[Theorem 7.5]{joyce-song}).
\end{rem}

\subsection{Counting invariants}\label{counting-inv}
For $\tz \in \R^{\Q_0}$ and a finite dimensional $\A$-module $\V$ we set 
\[
\theta_{\tz}(\V)=\frac{\sum_{k\in\Q_0}\tz_k\cdot \dim \V_k}{\sum_{k\in\Q_0}\dim \V_k}.
\]
\begin{defn}
A finite dimensional $\A$-module $\V$ is said to be $\theta_{\tz}$-(semi)stable if we have
\[
\theta_{\tz}(\V')\,(\le)\,\theta_{\tz}(\V)
\]
for any nonzero proper $\A$-submodule $\V'$
\end{defn}
Here we adapt the convention for the short-hand notation. The above
means two assertions: semistable if we have `$\le$', and stable if we
have `$<$'.

\begin{rem}
\begin{enumerate}
\item[(1)] These stability conditions coincide with ones in geometric invariant theory (\cite{king}). 
\item[(2)] For a real number $c$, we put $\tz':=(\zeta_k+c)_{k\in \Q_0}\in\R^{\Q_0}$. 
Then we have
\[
\theta_{\tz'}(\V)=\theta_{{\tz}}(\V)+c.
\]
Hence $\theta_{\tz'}$-(semi)stability and $\theta_{\tz}$-(semi)stability are equivalent. 
\end{enumerate}
\end{rem}

\begin{thm}[\protect{\cite{rudakov}}]
Let a stability parameter $\tz\in\R^{\Q_0}$ be fixed.
\begin{enumerate}
\item[(1)] A finite dimensional $\A$-module $\V\in\hafmod$ has \HN:
\[
\V=\V_0\supset \V_1\supset \cdots \supset \V_k\supset \V_{k+1}=0
\]
such that $\V_i/\V_{i+1}$ is $\theta_{\tz}$-semistable for $i=0,1,\ldots,k$ and 
\[
\theta_{\tz}(\V_0/\V_{1})<\theta_{\tz}(\V_1/\V_2)<\cdots<\theta_{\tz}(\V_k/\V_{k+1}).
\]
The Harder-Narasimhan filtration is unique.
\item[(2)] A finite dimensional $\theta_{\tz}$-semistable $\A$-module $\V\in\hafmod$ has \JH:
\[
\V=\V_0\supset \V_1\supset \cdots \supset \V_k\supset \V_{k+1}=0
\]
such that $\V_i/\V_{i+1}$ is $\theta_{\tz}$-stable for $i=0,1,\ldots,k$ and 
\[
\theta_{\tz}(\V_0/\V_{1})=\theta_{\tz}(\V_1/\V_2)=\cdots=\theta_{\tz}(\V_k/\V_{k+1}).
\]
\end{enumerate}
\end{thm}

We sometimes denote an $\A$-module $(V,\C,i)$ with $1$-dimensional vector space at the vertex $\infty$ simply by $(V,i)$.
\begin{defn}\label{def-stability}
\begin{enumerate}
\item[(1)]
Given $\zeta\in\R^{Q_0}$, we take $\tz\in\R^{\Q_0}$ such that $\tz_k=\zeta_k$ for $k\in Q_0$.
A finite dimensional $A$-module $V$ said to be $\zeta$-(semi)stable if the $\A$-module $(V,0,0)$ is $\theta_{\tz}$-(semi)stable.
The definition does not depend on the choice of $\tz_\infty$.
\item[(2)]
Given $(V,i)\in \hafmod$ and $\zeta\in\R^{Q_0}$, 
we define $\tz\in\R^{\Q_0}$ by $\tz_k=\zeta_k$ for $k\in Q_0$ and
\[
\tz_\infty=-\zeta\cdot\dimv\,V.
\]
We say $(V,i)$ is $\zeta$-(semi)stable if it is $\theta_{\tz}$-(semi)stable.
\end{enumerate}
\end{defn}
\begin{lem}
An $\A$-module $(V,i)$ is $\zeta$-(semi)stable if the following conditions are satisfied:
\begin{enumerate}
\item[(A)] for any nonzero $A$-submodule $0\neq S\subseteq V$, we have
\[
\zeta\cdot\dimv\,S\,(\le)\,0,
\]
\item[(B)] for any proper $A$-submodule $T\subsetneq V$ such that $\mr{im}(i)\subset T_0$, we have
\[
\zeta\cdot\dimv\,T\,(\le)\,\zeta\cdot\dimv\,V.
\]
\end{enumerate}
\end{lem}

For $\zeta\in\R^{Q_0}$ and $\vv=(v_k)\in (\Z_{\geq 0})^{Q_0}$, let $\M{\zeta}{\vv}$ (resp. $\mf{M}_{\zeta}^{\,\mr{s}}(\vv)$) denote the moduli space of $\zeta$-semistable (resp. $\zeta$-stable) $\A$-modules $(V,i)$ such that $\dimv\,V=\vv$.
They are constructed using geometric invariant theory (\cite{king}). 

A stability parameter $\zeta\in\R^{Q_0}$ is said to be generic if $\zeta$-semistability and $\zeta$-stability are equivalent to each other. 
Since the defining relation of $A$ is derived from the derivation of the superpotential, the moduli space $\M{\zeta}{\vv}$ has a symmetric perfect obstruction theory (\cite[Theorem 1.3.1]{szendroi-ncdt}). 
By the result of \cite{behrend-dt} a constructible $\Z$-valued function $\nu$ is defined on the moduli space $\M{\zeta}{\vv}$. 
We define the counting invariants 
\[
D^{\mathrm{eu}}_{\zeta}(\vv):=\chi(\M{\zeta}{\vv}),\quad 
D_{\zeta}(\vv):=\sum_{n\in\Z}n\cdot \chi(\nu^{-1}(n))
\]
where $\chi(-)$ denote topological Euler numbers.
We encode them into the generating functions
\[
\mca{Z}^{\mathrm{eu}}_{\zeta}(\q):=\sum_{\vv\in(\Z_{\geq 0})^{Q_0}}D^{\mathrm{eu}}_{\zeta}(\vv)\cdot\q^\vv,\quad 
\mca{Z}_{\zeta}(\q):=\sum_{\vv\in(\Z_{\geq 0})^{Q_0}}D_{\zeta}(\vv)\cdot\q^\vv
\]
where $\q^\vv=\prod_{k\in ^{Q_0}}q_k^{v_k}$ and $q_k$'s are formal variables.

In the rest of this section, we will work on $\mca{Z}^{\mathrm{eu}}_{\zeta}(\q)$. In \S \ref{mutation}, we will compare $\mca{Z}_{\zeta}(\q)$ with $\mca{Z}^{\mathrm{eu}}_{\zeta}(\q)$.
\begin{NB}
The $3$-dimensional torus $T$ acts on $\mca{Y}$, and so it also acts on $\M{\zeta}{\vv}$.  
Assume that the set of $T$-fixed points $\M{\zeta}{\vv}^T$ is isolated and finite (see \S \ref{subsec-DTPTNCDT} and \S \ref{subsec-isolated}). 
By the argument in the proof of \cite[Lemma 2.5.2, Corollary 2.5.3 and Theorem 2.7.1]{szendroi-ncdt}, the contribution of a $T$-fixed point $P\in\M{\zeta}{\vv}^T$ to $D_\zeta(\vv)$ is 
\[
(-1)^{v_0+\sum_{a\in Q_1}v_{\mr{out}(a)}v_{\mr{in}(a)}-\sum_{k\in Q_0}v_{k}^2}=(-1)^{\sum_{k\in \hI _r}v_{k}+\sum_{k\in I}v_{k}},
\]
and so we have
\begin{align}\label{z}
\mca{Z}_{\zeta}(\q)&=\sum_{\vv\in(\Z_{\geq 0})^{Q_0}}(-1)^{\sum_{k\in \hI _r}v_{k}+\sum_{k\in I}v_{k}}\left|\,\M{\zeta}{\vv}^T\right|\cdot\q^\vv\\
&=\mca{Z}^{\mathrm{eu}}_{\zeta}(\mathbf{p})\notag
\end{align}
where the new formal variables $\mathbf{p}$ are given by
\[
p_k=
\begin{cases}
q_k & k\neq 0, k\in\hat{I}_r,\,\text{or},\, k=0, k\neq\hat{I}_r,\\
-q_k & k\neq 0, k\neq\hat{I}_r,\,\text{or},\, k=0, k\in\hat{I}_r.
\end{cases}
\]
\end{NB}

\subsection{Wall-crossing formula}\label{subsec.wall-crossing}
The set of non-generic stability parameters is the union of the hyperplanes in $\R^{Q_0}$. 
Each hyperplane is called a wall and each connected component of the set of generic parameters is called a chamber. 
The moduli space $\M{\zeta}{\vv}$ does not change as long as $\zeta$ moves in a chamber. 

Let $\zeta^\circ=(\zeta_k)_{k\in Q_0}$ be a stability parameter on a single wall and set $\zeta^\pm=(\zeta_k\pm\varepsilon)_{k\in Q_0}$ for sufficiently small $\varepsilon>0$.
\begin{prop}\label{prop3.8.}
\begin{enumerate}
\item[(1)]
Let $\V'=(V',\C,i')$ be a 
$\zeta^+$-stable $\A$-module.
Then we have an exact sequence 
\[
0\to \V\to \V'\to \V''\to 0,
\]
where $\V=(V,\C,i)$ is a 
$\zeta^\circ$-stable $\A$-module, $\V''=(V'',0,0)$ and $V''$ is a 
${\zeta^\circ}$-semistable $A$-module.
The isomorphism class of $\V$ and $V''$ are determined uniquely. 
In particular, assume $\V'=(V',\C,i')$ is $T$-invariant then $\V$ and $V''$ are also $T$-invariant.
\item[(2)]
Let $\V'=(V',\C,i')$ be a $\zeta^-$-stable $\A$-module.
Then we have an exact sequence
\[
0\to \V''\to \V'\to\V \to 0,
\]
where $\V=(V,\C,i)$ is a 
$\zeta^\circ$-stable $\A$-module, $\V''=(V'',0,0)$ and $V''$ is a 
${\zeta^\circ}$-semistable $A$-module.
The isomorphism class of $\V$ and $V''$ are determined uniquely. 
In particular, assume $\V'=(V',\C,i')$ is $T$-invariant then $\V$ and $V''$ are also $T$-invariant.
\end{enumerate}
\end{prop}
\begin{proof}
We take $\tilde{\zeta}^\circ\in \R^{\Q_0}$ as in Definition \ref{def-stability}.
Let
\[
\V'=\V_0\supset\cdots\supset \V_{M}\supset \V_{M+1}=0
\]
be \JH of $\V'$ with respect to the $\theta_{\tz^\circ}$-stability.
Since $\dim \V'_\infty=1$, there is an integer $0\leq m\leq M$ such that $\dim(\V_m/\V_{m+1})_\infty=1$ and $\dim(\V_{m'}/\V_{m'+1})_\infty=0$ for any $m'\neq m$. 
Then for $m'\neq m$ we have 
\[
\zeta^+\cdot\dim(\V_{m'}/\V_{m'+1})=\varepsilon\cdot\sum_{k\in I}(\V_{m'}/\V_{m'+1})_k>0.
\]
From the $\zeta^+$-stability of $\V'$, we have $m=M$. 
Put $\V=\V_M$ and $\V''=\V'/\V_M$, we have the required sequence.

Let 
\[
\V''=\V''_0\supset\cdots\supset \V''_{M''}\supset \V_{M''+1}=0
\]
be the \HN of $\V''$ with respect to the $\theta_{\tz^-}$-stability. 
Since $\V''$ is $\theta_{\tz^\circ}$-semistable, we have
\[
\zeta^-\cdot\dim(\V''_{M''})<\zeta^\circ\cdot\dim(\V''_{M''})\leq \zeta^\circ\cdot\dim(\V'')=0<\zeta^-\cdot\dim(\V).
\]
Then the sequence
\[
\V'=\pi^{-1}(\V''_0)\supset\cdots\supset \pi^{-1}(\V''_{M''+1})=\V \supset 0
\]
gives the \HN of $\V'$ with respect to the $\theta_{\tz^-}$-stability, 
where $\pi\colon \V'\twoheadrightarrow \V''$ is the projection.
Assume $\V'$ is $T$-invariant, then it follows from the uniqueness of \HN that $\V$ and $\V''$ are $T$-invariant.  

We can verify the claim of (2) similarly.
\end{proof}

We identify $\Z^{Q_0}$ with the root lattice of affine Lie algebra of type $A_N$ and denote the set of positive root vectors by $\rs^+$. 
We put $W_\alpha:=\{\zeta\in\Z^{Q_0}\mid\zeta\cdot \alpha=0\}$ ($\alpha\in \rs^+$). 
\begin{prop}\label{prop3.9.}
Assume $C$ is a 
$\zeta$-stable $A$-module for some $\zeta\in\R^{Q_0}$. 
Then $\dimv\,C\in\rs^+$. 
Moreover, given a positive real root $\alpha\in\rs^{\mr{re},+}$ and a stability parameter $\zeta^\circ$ such that $\zeta^\circ$ is on the wall $W_\alpha$ but not on any other wall $W_{\alpha'}$ ($\alpha'\neq\alpha$),
then we have the unique ${\zeta^\circ}$-stable $T$-invariant $A$-module $C$ such that $\dimv\,C=\alpha$.  
\end{prop}
\begin{proof}
Note that we have the natural homomorphism
\[
R\to \bigoplus_{k\in \hI}\End_{\mca{Y}}(L_k)\to \End_{\mca{Y}}(\oplus L_k)\simeq A,
\]
where the first one is the diagonal embedding. 
The image of this map is central subalgebra of $A$.
Any $A$-module has the $R$-module structure given by this homomorphism. 
Any finite dimensional $A$-module is supported on finite number of points on $\mathrm{Spec}(R)=\mca{X}$ and so any finite dimensional $A$-module $C$ which is $\zeta$-stable for some $\zeta$ is supported on a maximal ideal $\mathcal{I}\subset R$.
Any nonzero element of $\mca{I}$ induces an $A$-module automorphism on $C$ by the multiplication.     
Since $C$ is $\zeta$-stable, any $A$-module automorphism on $C$ is either zero or isomorphic. 
But this can not be isomorphic, because $C$ is supported on $\mca{I}$. 
Hence we have $\mca{I}\cdot C=0$. 

Suppose that $\mathcal{I}$ corresponds to a nonsingular point $P$ on $\mca{X}$ and $\mca{I}\cdot C=0$. 
Then $C$ is, as a complex of sheaves on $\mca{Y}$, supported on $\pi^{-1}(P)$. 
Then we can verify that $C=\mca{O}_{\pi^{-1}(P)}$.
The singular points on $\mca{X}$ is classified as follows:
\begin{itemize}
\item the unique $T$-invariant $(X,Y,Z,W)$,
\item $(X,Y,Z-a,W)$ ($a\neq 0$) or
\item $(X,Y,Z,W-b)$ ($b\neq 0$).
\end{itemize}
An $A$-module $C$ such that $(X,Y,Z,W)\cdot C=0$ is a module over the preprojective algebra of type $\hat{A}_N$. 
The dimension of the moduli space of representations of the preprojective algebra of type $\hat{A}_N$ is $\vv\mb{C}\vv$, where $\mb{C}$ is the Cartan matrix of type $\hat{A}_N$ (\cite{quiver1}). 
If there exists a $\zeta$-stable $A$-module $C$ with $\dimv\, C=\mathbf{v}$, then $\vv\mb{C}\vv$ is nonnegative, which is the definition of the root vectors. 
The uniqueness follows from the irreducibility of the moduli space (\cite{cb-decomp}).  
Let $C$ be an $A$-module such that $(X,Y,Z-a,W)\cdot C=0$ ($a\neq 0$).  
For $k\in\hat{I}$ such that $\sigma_y(k+\frac{1}{2})=1$, $h^-_{k+\frac{1}{2}}$ and $h^+_{k+\frac{1}{2}}$ give isomorphisms between $C_k$ and $C_{k+1}$. 
For $k\in\hat{I}$ such that $\sigma_y(k+\frac{1}{2})=0$, we have 
\[
h^+_{k+\frac{1}{2}}\circ h^-_{k+\frac{1}{2}}=0,\quad h^-_{k+\frac{1}{2}}\circ h^+_{k+\frac{1}{2}}=0.
\]
Thus we can identify an $A$-modules $C$ such that $(X,Y,Z-a,W)\cdot C=0$ ($a\neq 0$) with a module over the preprojective algebra of type $\hat{A}_{N_0}$.
Under this identification, a root vector of the root system of type $\hat{A}_{N_0}$ corresponds to a root vector of the root system of type $\hat{A}_{N}$.
Hence for a $\zeta$-stable $A$-module $C$ such that $(X,Y,Z-a,W)\cdot C=0$ its dimension vector is a root vector of the root system of type $\hat{A}_{N}$.
Similarly, for a $\zeta$-stable $A$-module $C$ such that $(X,Y,Z,W-b)\cdot C=0$ its dimension vector is a root vector of the root system of type $\hat{A}_{N}$. 
\end{proof}
\begin{rem}
Note that the fiber over the point corresponds to the maximal ideal $(X,Y,Z-a,W)$ (resp. $(X,Y,Z,W-b)$) is the $A_{N_0}$ (resp. $A_{N_1}$) configuration of $\CP^1$'s. 
\end{rem}
\begin{cor}
The set of nongeneric parameters is the union of the hyperplanes 
$W_\alpha$ ($\alpha\in \rs^+$). 
\end{cor}
Take a positive real root $\alpha\in\rs^{\mr{re},+}$ and a parameter $\zeta^\circ\in\R^{Q_0}$ which is on $W_\alpha$ but not on any other wall. 
Let $C$ be the unique $T$-invariant $\zeta^\circ$-stable $A$-module such that $\zeta^\circ\cdot\dimv\,C$. 
We set $\zeta^\pm=(\zeta_k\pm \varepsilon)$ for sufficiently small $\varepsilon$. 
We fix these notations throughout this subsection. 
\begin{prop}[\protect{\cite[Proposition 3.7]{nagao-nakajima}}]\label{prop3.11.} 
For a $\zeta^\circ$-stable $\A$-module $\V=(V,i)$ we have
\[
\ext^1_{\A}(C,\V)-\ext^1_{\A}(\V,C)=\dim C_0.
\]
\end{prop}
\begin{proof}
Since $C$ and $\V$ are $\zeta^\circ$-stable and not isomorphic each other, we have $\hom_{\A}(C,\V)=\hom_{\A}(\V,C)=0$. So the claim follows form Proposition \ref{prop3.1.}.
\end{proof}
\begin{prop}\label{prop3.12.}
\begin{itemize}
\item[(1)]
\[
\ext^1_A(C,C)=\begin{cases}
0 & \text{$\sum_{k\notin \hI _r}\alpha_k$ is odd},\\
1 & \text{$\sum_{k\notin \hI _r}\alpha_k$ is even}.
\end{cases}
\]
\item[(2)]
Assume $\sum_{k\notin \hI _r}\alpha_k$ is even. 
For a positive integer $m$, we have the unique indecomposable $A$-module $C_m$ which is described as $m-1$ times successive extensions of $C$'s.
\end{itemize}
\end{prop}
We prove this proposition in \S \ref{appendix}.

Let $\alpha\in\rs^{\mr{re},+}$ be a positive real root such that $\sum_{k\notin \hI _r}\alpha_k$ is odd. 
In such cases, wall-crossing formulas are given in \cite{nagao-nakajima}:
\begin{thm}[\protect{\cite[Theorem 3.9]{nagao-nakajima}\label{thm3.13.}}]
\[
\mca{Z}^{\mathrm{eu}}_{\zeta^-}(\q)=\left(1+\q^{\alpha}\right)^{\alpha_0}\cdot \mca{Z}^{\mathrm{eu}}_{\zeta^+}(\q).
\]
\end{thm}
\begin{rem}
To be precise we should modify the argument in \cite{nagao-nakajima} a little, since the stable objects on the wall are not unique, while so are the $T$-invariant stable objects. 
See the argument after Proposition \ref{prop3.15.} and Remark \ref{rem_non_T_inv}.
\end{rem}
\begin{prop}\label{prop3.14.}
\begin{enumerate}
\item[(1)] 
Let $\V'=(V',\C,i')$ be a $T$-invariant $\zeta^+$-stable $\A$-module.
Then we have an exact sequence
\[
0\to \V\to \V'\to \oplus_{m'\geq 1}(C_{m'})^{\bigoplus n_{m'}}\to 0,
\]
where $\V=(V,\C,i)$ is a $T$-invariant $\zeta^\circ$-stable $\A$-module.
The integers $n_{m'}$ and isomorphism class of $\V$ are determined uniquely and satisfy 
\[
\hom(\V',C_m)=\sum_{m'\geq 1}n_{m'}\cdot\min(m',m).
\] 
Moreover, the composition of the maps
\[
\C^{N_m}\hookrightarrow \Hom_{\A}(C_{m},\oplus_{m'\geq 1}(C_{m'})^{\oplus n_{m'}})\longrightarrow\Ext^1_{\A}(C,\V)
\]
is injective. 
Here $N_m=\sum_{m'\geq m}n_{m'}$ and the first map is induced by inclusions $C_m\hookrightarrow C_{m'}$ ($m'\geq m$).
The second map is given by composing the inclusion $C\hookrightarrow C_m$ and $\V'\in\Ext^1_{\A}(\oplus_{m'\geq 1}(C_{m'})^{\oplus n_{m'}},\V)$.
\item[(2)]
Let $\V'=(V',\C,i')$ be a $T$-invariant $\zeta^-$-stable $\A$-module.
Then we have an exact sequence 
\[
0\to\oplus_{m'\geq 1}(C_{m'})^{\bigoplus n_{m'}} \to \V'\to\V \to 0,
\]
where $\V=(V,\C,i)$ is a $T$-invariant $\zeta^\circ$-stable $\A$-module. 
The integers $n_{m'}$ and isomorphism class of $\V$ are determined uniquely and satisfy
\[
\hom(C_m,\V')=\sum_{m'\geq 1}n_{m'}\cdot\min(m',m).
\] 
Moreover, the composition of the maps
\[
\C^{N_m}\hookrightarrow \Hom_{\A}(\oplus_{m'\geq 1}(C_{m'})^{\oplus n_{m'}},C_m)\longrightarrow\Ext^1_{\A}(\V,C)
\]
is injective.
Here $N_m=\sum_{m'\geq m}n_{m'}$ and the first map is induced by surjections $C_{m'}\twoheadrightarrow C_{m}$ ($m'\geq m$).
The second map is given by composing the surjection $C_m\twoheadrightarrow C$ and $\V'\in\Ext^1_{\A}(\V,\oplus_{m'\geq 1}(C_{m'})^{\oplus n_{m'}})$.
\end{enumerate}
\end{prop}
\begin{proof}
The existence of the sequences follow from Proposition \ref{prop3.8.}, Proposition \ref{prop3.9.} and Proposition \ref{prop3.12.}. 
Since $\hom(\V,C_m)=0$, we have $\hom(\V',C_m)=\hom(\oplus_{m'}C_{m'},C_m)$ from the long exact sequences. 
The middle equations follow from $\hom(C_{m'},C_m)=\min(m',m)$.
The compositions of the maps are injective since otherwise $\V'$ has $C$ as its direct summand. 
\end{proof}

Given a non-increasing sequence $(N_m)_{m\geq 1}$ of non-negative integers such that $N_{m''}=0$ for some $m''$, let $Fl((N_k);N)$ be the flag variety 
\[
\{0= W_{m''}\subseteq \cdots \subseteq W_{1}\subseteq \C^N\mid \dim W_m=N_m\}.
\] 
We can verify the following claim as well:
\begin{prop}\label{prop3.15.}
\begin{enumerate}
\item[(1)] 
Let $\V=(V,i)$ be a $T$-invariant $\zeta^\circ$-stable $\A$-module. 
For an element
\[
(W_k)\in Fl((N_k);\ext^1_{\A}(C,\V))^T,
\]
let $\V'$ denote the $\A$-module given by the universal extension 
\[
0\to \V\to \V'\to \bigoplus_{m\geq 1}(C_{m})^{\oplus n_{m}}\to 0,
\]
such that the image of the composition map in Proposition \ref{prop3.14.} coincides with $W_k$.
Then $\V'$ is $T$-invariant and $\zeta^+$-stable.
\item[(2)]
Let $\V=(V,i)$ be a $T$-invariant $\zeta^\circ$-stable $\A$-module. 
For an element
\[
(W_k)\in Fl((N_k);\ext^1_{\A}(\V,C))^T,
\]
let $\V'$ denote the $\A$-module given by the universal extension 
\[
0\to \bigoplus_{m\geq 1}(C_{m})^{\oplus n_{m}} \to \V'\to \V\to 0
\]
such that the image of the composition map in Proposition \ref{prop3.14.} coincides with $W_k$.
Then $\V'$ is $T$-invariant and $\zeta^-$-stable.
\end{enumerate}
\end{prop}
Hereafter we denote the set of $T$-fixed points on $X$ by ${}^TX$. 
Let $R(T)$ be the representation ring of $T$. 
For a nonnegative integer $N$ (resp. $\mca{N}\in R(T)$), 
let ${}^T\mf{M}^{\,\mr{s}}_{\zeta^\circ}(\vv)_N$ (resp. ${}^T\mf{M}^{\,\mr{s}}_{\zeta^\circ}(\vv)_{\mca{N}}$) denote the subscheme of ${}^T\mf{M}^{\,\mr{s}}_{\zeta^\circ}(\vv)$ consisting of closed points $\V$ such that $\ext^1(C,\V)=N$ (resp. $\Ext^1(C,\V)=\mca{N}$ as $T$-modules).
Let ${}^T\M{\zeta^+}{\vv'}_{(n_m)}$ denote the subscheme of ${}^T\M{\zeta^+}{\vv'}$ consisting of closed points $\V'$ such that 
$\hom(\V',C_m)=\sum_{m'\geq 1}n_{m'}\cdot\min(m',m)$
We have the canonical morphism ${}^T\M{\zeta^+}{\vv'}_{(n_m)}\to {}^T\mf{M}^{\,\mr{s}}_{\zeta^\circ}(\vv)$ where $\vv=\vv'-\sum_{m}mn_m\cdot\dimv\,C$ such that a closed point $\V'\in{}^T\M{\zeta^+}{\vv'}_{(n_m)}$ is mapped to the closed point $\V\in{}^T\mf{M}^{\,\mr{s}}_{\zeta^\circ}(\vv)$ appeared in the exact sequence
\[
0\to \V\to \V'\to \oplus_{m'\geq 1}(C_{m'})^{\bigoplus n_{m'}}\to 0.
\]
Let ${}^T\M{\zeta^+}{\vv'}_{(n_m),N}$ (resp. ${}^T\M{\zeta^+}{\vv'}_{(n_m),\mca{N}}$) denote the inverse image of ${}^T\mf{M}^{\,\mr{s}}_{\zeta^\circ}(\vv)_N$ (resp. ${}^T\mf{M}^{\,\mr{s}}_{\zeta^\circ}(\vv)_{\mca{N}}$) with respect to the above morphism. 
Similarly, we define ${}^T\mf{M}^{\,\mr{s}}_{\zeta^\circ}(\vv)^N$, ${}^T\mf{M}^{\,\mr{s}}_{\zeta^\circ}(\vv)^{\mca{N}}$,  ${}^T\M{\zeta^+}{\vv'}^{(n_m)}$, ${}^T\M{\zeta^+}{\vv'}^{(n_m),N}$ and ${}^T\M{\zeta^+}{\vv'}^{(n_m),\mca{N}}$.

By Proposition \ref{prop3.14.} and Proposition \ref{prop3.15.}, 
the natural map 
\[
{}^T\M{\zeta^+}{\vv'}_{(n_m),\mca{N}}\to {}^T\mf{M}^{\,\mr{s}}_{\zeta^\circ}(\vv)_\mca{N}
\]
is a fibration.
So we have
\begin{align*}
\sum_{\vv'}\chi({}^T\M{\zeta^+}{\vv'})\cdot\q^{\vv'}&=
\sum_{\vv',(n_m),\mca{N}}\chi({}^T\M{\zeta^+}{\vv'}_{(n_m),\mca{N}})\cdot\q^{\vv'}\\
&=\sum_{\vv,(n_m),\mca{N}}\chi(Fl((N_m);N))\cdot\chi({}^T\mf{M}^{\,\mr{s}}_{\zeta^\circ}(\vv)_\mca{N})\cdot\q^{\vv+\sum mn_m\cdot\dimv(C)}\\
&=\sum_{\vv,N}\left(\sum_{(n_m)}\chi(Fl((N_m);N))\cdot\q^{\sum mn_m\cdot\dimv(C)}\right)\chi({}^T\mf{M}^{\,\mr{s}}_{\zeta^\circ}(\vv)_{N})\cdot\q^\vv\\
&=\sum_{\vv,N}\left(1-\q^{\dimv(C)}\right)^{-N}\chi({}^T\mf{M}^{\,\mr{s}}_{\zeta^\circ}(\vv)_{N})\cdot\q^\vv.
\end{align*}
Similarly we have
\[
\sum_{\vv'}\chi({}^T\M{\zeta^-}{\vv'})\cdot\q^{\vv'}
=\sum_{\vv,N}\left(1-\q^{\dimv(C)}\right)^{-N}\chi({}^T\mf{M}^{\,\mr{s}}_{\zeta^\circ}(\vv)^{N})\cdot\q^\vv.
\]
By Proposition \ref{prop3.11.} we have $\M{\zeta}{\vv}_N=\M{\zeta}{\vv}^{N+\dim C_0}$. Hence we have 
\begin{align*}
\sum_{\vv'}\chi(\M{\zeta^-}{\vv'})\cdot\q^{\vv'}
&=\sum_{\vv,N}\left(1-\q^{\dimv(C)}\right)^{-N}\chi(\mf{M}^{\,\mr{s}}_{\zeta^\circ}(\vv)^{N})\cdot\q^\vv\\
&=\sum_{\vv,N}\left(1-\q^{\dimv(C)}\right)^{-N}\chi(\mf{M}^{\,\mr{s}}_{\zeta^\circ}(\vv)_{N-\dim C_0})\cdot\q^\vv\\
&=\sum_{\vv,N}\left(1-\q^{\dimv(C)}\right)^{-N-\dim C_0}\chi(\mf{M}^{\,\mr{s}}_{\zeta^\circ}(\vv)_{N})\cdot\q^\vv\\
&=\left(1-\q^{\dimv(C)}\right)^{-\dim C_0}\cdot\sum_{\vv'}\chi(\M{\zeta^+}{\vv'})\cdot\q^{\vv'}.
\end{align*}
\begin{rem}\label{rem_non_T_inv}
In the argument above, we do not use any module on a wall which is not $T$-invariant.
\end{rem}
In summary, we have the following {\it wall-crossing formula}: 
\begin{thm}\label{thm3.16.}
We have
\[
\mca{Z}^{\mathrm{eu}}_{\zeta^-}(\q)=\left(1-\varepsilon(\alpha)\q^{\alpha}\right)^{-\varepsilon(\alpha)\alpha_0}\cdot \mca{Z}^{\mathrm{eu}}_{\zeta^+}(\q),
\]
where 
\[
\varepsilon(\alpha)=(-1)^{\sum_{k\notin \hI _r}\alpha_k}.
\]
\end{thm}

\begin{NB}
\subsection{DT, PT and NCDT}\label{subsec-DTPTNCDT}
Note that the set $\{[\mca{C}_k]\}_{k\in I}$ form a basis of $H_2(Y;\Z)$. We identify $H_2(Y)$ with $\Z^I$.  
For $n\in\Z_{\geq 0}$ and $\beta\in H_2(Y)$, let $I_n(Y,\beta)$ denote the moduli space of ideal sheaves $\mca{I}_{Z}$ of one dimensional subschemes $\OO_Z\subset\OO_Y$ such that $\chi(\OO_Z)=n$ and $[Z]=\beta$. 
Note that the set of torus fixed points $I_n(Y,\beta)^T$ is isolated and finite (\cite{mnop}).
We define the {\it Donaldson-Thomas invariants} $I_{n,\beta}$ from $I_n(Y,\beta)$ using Behrend's function as is \S \ref{counting-inv} (\cite{thomas-dt}, \cite{behrend-dt}), 
and their generating function by
\[
\mca{Z}_{\mr{DT}}(Y;q,\mb{t}):=\sum_{n,\beta}I_{n,\beta}\cdot q^n\mb{t}^\beta
\] 
where $\mb{t}^\beta=\prod_{k\in ^{I}}t_k^{\beta_k}$ and $t_k$ ($k\in I$) is a formal variable.

Let $P_n(Y,\beta)$ denote the moduli space of stable pairs $(F,s)$ such that $\chi(F)=n$ and $[\mr{supp}(F)]=\beta$. 
The set of torus fixed points $I_n(Y,\beta)^T$ is also isolated and finite (\cite{pt3}).
We define the {\it Pandharipande-Thomas invariants} $P_{n,\beta}$ 
and their generating function $\mca{Z}_{\mr{PT}}(Y;q,\mb{t})$ similarly (\cite{pt1}).

We set $\zeta^\circ=(-N+1,1,\ldots,1)$, $\zeta^\pm=(-N+1\pm\varepsilon,1,\ldots,1)$ and $q=q_0\cdot q_1\cdot\cdots\cdot q_{N-1}$.    
By the result in \cite[\S 2]{nagao-nakajima} and the fact that the sets of torus fixed points are isolated and finite, we have the following theorem:
\begin{prop}
\[
\mca{Z}_{\mr{DT}}(Y;q,q_1^{-1},\ldots,q_{N-1}^{-1})=\mca{Z}_{\zeta^-}(\q),\quad
\mca{Z}_{\mr{PT}}(Y;q,q_1^{-1},\ldots,q_{N-1}^{-1})=\mca{Z}_{\zeta^+}(\q).
\]
\end{prop}

Let $\zeta_{\mathrm{triv}}$ be a parameter such that $(\zeta_{\mathrm{triv}})_k>0$ for any $k$.
Note that $\M{\zeta_{\mathrm{triv}}}{\vv}$ is empty unless $\vv=0$ and so $\mca{Z}_{\zeta_{\mathrm{triv}}}(\q)=1$.
Let $\zeta_{\mathrm{cyclic}}$ be a parameter such that $(\zeta_{\mathrm{cyclic}})_k<0$ for any $k$.
The invariants $D_{\zeta_{\mathrm{cyclic}}}(\vv)$ are the {\it non-commutative Donaldson-Thomas invariants} defined in \cite{szendroi-ncdt}.
Note that the set of torus fixed points $\mathfrak{M}_{\zeta_{\mathrm{cyclic}}}(\vv)^T$ is isolated and finite.
We denote their generating function $\mca{Z}_{\zeta_{\mathrm{cyclic}}}(\q)$ by $\mca{Z}_{\mr{NCDT}}(\q)$.

We divide the set of positive real roots into the following two parts: 
\[
\rs^{\mr{re},+}_\pm=\{\alpha\in\rs^{\mr{re},+}\mid \pm\zeta^\circ\cdot\alpha <0\}.
\]
Applying the wall-crossing formula and comparing the equations \eqref{z} in \S \ref{counting-inv}, we obtain the following relations between generating functions:
\begin{thm}
\begin{align*}
&\mca{Z}_{\mr{NCDT}}(\q)=\left(\prod_{\alpha\in\Delta^{\mr{re},+}_-} (1+(-1)^{\alpha_0}\mb{q}^\alpha)^{\varepsilon(\alpha)\alpha_0}\right)\cdot\mca{Z}_{\mr{DT}}(Y;q,q_1^{-1},\ldots,q_{N-1}^{-1}),\\
&\mca{Z}_{\mr{PT}}(Y;q,q_1^{-1},\ldots,q_{N-1}^{-1})=\prod_{\alpha\in\Delta^{\mr{re},+}_+} (1+(-1)^{\alpha_0}\mb{q}^\alpha)^{\varepsilon(\alpha)\alpha_0}.
\end{align*}
\end{thm}

\begin{rem}
For the case $N_1=0$, the formula on $\mca{Z}_{\mr{NCDT}}$ and $\mca{Z}_{\mr{DT}}$ have been given in \cite{young-mckay}.
\end{rem}

We define the sets of positive real roots of the finite root system by 
\[
\rs^{\mr{fin},+}=\{\alpha_{[a,b]}:=\alpha_a+\cdots+\alpha_b \mid 0<a\leq b<N\},
\]
then we have
\[
\rs^{\mr{re},+}_+=\{\alpha+n\delta \mid \alpha\in \rs^{\mr{fin},+},\ n\geq 0\}.
\]
Note that for $\alpha\in\rs$ we have $\varepsilon(\alpha+n\delta)=\varepsilon(\alpha)$. 
Let 
\[
M(x,q)=\prod_{n=1}^\infty(1-xq^n)^{-n}
\]
be the {\it MacMahon function}. 
\begin{cor}
\[
\mca{Z}_{\mr{PT}}(Y;q,q_1,\ldots,q_{N-1})=\prod_{0<a\leq b<N}M(q_{[a,b]},-q)^{\varepsilon(\alpha_{[a,b]})},
\]
where $q_{[a,b]}=q_a\cdot\cdots\cdot q_b$.
\end{cor}
The root lattice of the finite root system is identified with $H_2(Y)$ so that $\alpha_k$ corresponds to $[\mca{C}_k]$. 
The corollary claims the {\it Gopakumar-Vafa BPS state counts} in genus $g$ and class $\alpha$ {\it defined} in \cite[\S 3.4]{pt1} is given by 
\[
n_{g,\alpha}=\begin{cases}
-\varepsilon(\alpha_{[a,b]}) & \alpha=\alpha_{[a,b]},\, g=0,\\
0 & \text{otherwise}.
\end{cases}
\]
\end{NB}

\subsection{Appendix}\label{appendix}
In this subsection, we prove Proposition \ref{prop3.12.}.
First, recall the proof of Proposition \ref{prop3.9.}.
We can take integers $n_0$, $n_1$ and a basis $\{v_n\}_{n_0\leq n\leq n_1}$ of $C$ such that 
\[
C_k=\bigoplus_{\begin{subarray}{c}n_0\leq n\leq n_1,\\n\equiv k\ (\mr{mod}\, N)\end{subarray}}\C v_n,
\]
\begin{align*}
\bigr(h_i^+(v_n),h_i^-(v_{n+1})\bigl)&=(v_{n+1},0),\,\text{or}\ (0,v_n)\quad (i-\h\equiv n),\\
h_i^-(v_{n_0})&=0,\quad (i-\h\equiv n_0),\\
h_i^+(v_{n_1})&=0,\quad (i-\h\equiv n_1),\\
r_k(v_{n})&=0,\quad (k\equiv n,\, k\in \hI _r).
\end{align*}
Here the choice in the first equation depends on the stability. 
\begin{ex}\label{ex_1}
We put $N=4$, $|\hat{I}_r|=0$, $n_0=0$, $n_1=10$ and take $\theta$ such that
\[
0<\theta(0)\ll\theta(1)\ll\theta(2),\quad \theta(3)=-3(\theta(0)+\theta(1)+\theta(2))/2.
\]
For $n_0<j<n_1$, we put 
\[
C(j)=
\begin{cases}
+ & h^+_j(v_{j-1/2})=v_{j+1/2},\\
- & h^-_j(v_{j+1/2})=v_{j-1/2}.
\end{cases}
\]
Then we have 
\[
(C(1/2),\ldots, C(19/2))=(+,+,+,-,-,-,+,-,-,-).
\]
\begin{figure}[htbp]
  \centering
  \input{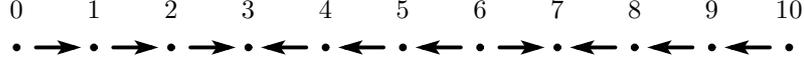}
  \caption{example of $C$}
  \label{fig:8}
\end{figure}
\end{ex}
\begin{lem}\label{lem_2.25}
Assume that we have $d\in\Z$ and $j<j'\in\tZ$ such that $n_0\leq j+\h,j+dN+\h$ and $j'-\h,j'+dN-\h\leq n_1$.
If we have $v_{j+\h}\notin\im(h^+_{\uj})$, $v_{j+\h+dN}\in\ker(h^-_{\uj})$, 
$v_{j'-\h}\notin\im(h^-_{\uj'})$ and $v_{j'-\h+dN}\in\ker(h^+_{\uj'})$, then $d=0$ and $\{j,j'\}=\{n_0-\h,n_1+\h\}$.
\end{lem}
\begin{proof}
Assume that $(j,j')\neq (n_0-\h,n_1+\h)$ or $d\neq 0$.
Note that
$\bigoplus_{j<n<j'}\C v_n$ 
is a nonzero proper $A$-submodule of $C$ and 
$\bigoplus_{j+dN<n<j'+dN}\C v_n$ 
is a nonzero proper $A$-quotient module of $C$. 
By the $\theta$-stability of $C$ we have
\[
\sum_{j<n<j'}\theta_{\underline{n}}<0, \quad \sum_{j+dN<n<j'+dN}\theta_{\underline{n}}> 0.
\]
This is a contradiction.
\end{proof}

Using the Koszul complex, we can compute $\Ext_A^1(C,C)$ as the cohomology of the following complex:
\[
\bigoplus_{a\in Q_1}\Hom(C_{\mr{out}(a)},C_{\mr{in}(a)})\overset{d_2}{\longrightarrow}\bigoplus_{b\in Q_1}\Hom(C_{\mr{in}(b)},C_{\mr{out}(b)})\overset{d_1}{\longrightarrow}\bigoplus_{k\in Q_0}\Hom(C_k,C_k).
\]

For $d\in\Z$ and $j\in\tZ$ such that both  $j$ and $j-\h+dN$ (resp. $j+\h+dN$) are in the interval $[n_0,n_1]$ we define an element
\[
\alpha^\pm_{j,d}\in \Hom\left(C_{\mr{in}(h^\pm_{\uj})},C_{\mr{out}(h^\pm_{\uj})}\right)
\]
by 
\[
\alpha^\pm_{j,d}(v_{n})=\delta_{n,j\pm\h}v_{j\mp\h+dN},
\]
and for $l\in\hat{I}_r$ such that both $l$ and $l+dN$ are in the interval $[n_0,n_1]$ we define an element
\[
\beta_{l,d}\in \Hom\left(C_{\mr{in}(r_{\ul})},C_{\mr{out}(r_{\ul})}\right)
\]
by 
\[
\beta_{l,d}(v_{n})=\delta_{n,l}v_{l+dN}.
\]

We define a set $J_0$ by 
\[
J_0=\{n\in\Z\mid n_0\leq n\leq n_1,\,\underline{n}\in \hI _r\}
\sqcup \{j\in \tZ\mid n_0<j<n_1\}.
\]
We define a set $J_d$ for $d\in\Z_{>0}$ such that $dN\leq n_1-n_0$ by
\[
\{n\in\Z\mid n_0\leq n\leq n_1-dN,\,\underline{n}\in \hI _r\}
\sqcup \{j\in \tZ\mid n_0<j<n_1-dN\}
\sqcup \{n_0-1/2\}
\]
if $h_{\underline{(n_0-1/2)}}^+(v_{n_0-1+dN})=0$, and 
\[
\{n\in\Z\mid n_0\leq n\leq n_1-dN,\,\underline{n}\in \hI _r\}
\sqcup \{j\in \tZ\mid n_0<j<n_1-dN\}
\sqcup \{n_1+1/2\}
\]
if $v_{n_0+1-dN}\notin \im\left(h_{\underline{(n_0+1/2)}}^+\right)$.
We also define a set $J_d$ for $d\in\Z_{<0}$ such that $-dN\leq n_1-n_0$ by
\[
\{n\in\Z\mid n_0-dN\leq n\leq n_1,\,\underline{n}\in \hI _r\}
\sqcup \{j\in \tZ\mid n_0-dN<j<n_1\}
\sqcup \{n_1+1/2\}
\]
if $h_{\underline{(n_1+1/2)}}^-(v_{n_1+1+dN})=0$, and 
\[
\{n\in\Z\mid n_0-dN\leq n\leq n_1,\,\underline{n}\in \hI _r\}
\sqcup \{j\in \tZ\mid n_0-dN<j<n_1\}
\sqcup \{n_0-1/2-dN\}
\]
if $v_{n_0-1-dN}\notin \im(h_{\underline{(n_0-1/2)}}^-)$.
Applying Lemma \ref{lem_2.25} for $j=n_0-\h$ and $j'=n_1+\h-dN$, we can see that these definitions make sense.

For $d\in\Z$ and $j\in\tZ\cap J_d$ we take an element $\beta_{j,d}$ in the kernel of $d_1$ as follows:
\begin{itemize}
\item $\beta_{j,d}:=\alpha^+_{j,d}$ 
if $v_{j+\h}\notin \im(h^+_{\uj})$ and
$v_{j-\h+dN}\in \ker(h^+_{\uj})$,
\item $\beta_{j,d}:=\alpha^-_{j,d}$ 
if $v_{j-\h}\notin \im(h^-_{\uj})$ and 
$v_{j+\h+dN}\in\ker(h^-_{\uj})$,
\item $\beta_{j,d}:=\alpha^+_{j,d}+\alpha^-_{j,d}$ 
if $h^-_{\uj}(v_{j+\h})=h^+_{\uj}(v_{j-\h+dN})=0$ or 
$h^+_{\uj}(v_{j-\h})=h^-_{\uj}(v_{j+\h+dN})=0$.
\end{itemize}
Note that the set
\[
\{\beta_{j,d}\}_{d\in \Z, j\in \tZ\cap J_d}\sqcup\{\beta_{l,d}\}_{d\in\Z,l\in \Z\cap J_d}
\]
forms a basis of $\ker(d_1)$. 
\begin{ex}
In Figure \ref{fig:9}, we describe $\beta_{9/2,-1}$, $\beta_{13/2,-1}$ and $\beta_{21/2,-1}$ in the case of Example \ref{ex_1}.
\begin{figure}[htbp]
  \centering
  \input{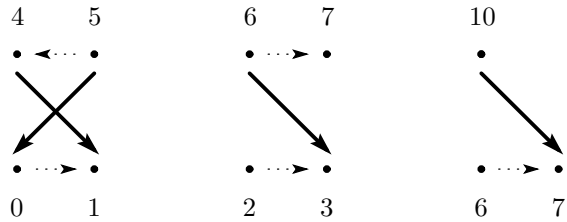}
  \caption{examples of elements in $\mr{ker}(d_1)$.}
  \label{fig:9}
\end{figure}
\end{ex}
\begin{NB}
For an element $a\in J_d$ let $a^+$ (resp. $a^-$) denote the minimal (resp. maximal) element in $J_d$ which is larger (resp. smaller) than $a$ if it exists. 
We divide $J_d$ into two disjoint subset $J'_d$ and $J''_d$ so that exactly one of $a$ or $a^+$ is in $J'_d$ for any $a\in J_d$ such that $a^+$ exists.
Furthermore we divide $J'_d$ (resp. $J''_d$) into disjoint subsets satisfying the following condition: 
Take $a\in J'_d$ such that $b:=(a^+)^+$ exists.
Let
$H$ be the subset of $J'_d$ containing $a$. 
Then $b\notin H$ if 
\begin{itemize}
\item $a\in \tZ$, $v_{a+1/2}\notin \im (h^+_{\underline{a}})$ and  $h^-_{\underline{a}}(v_{a+1/2+dN})=0$, or
\item $a^+\in \tZ$, $v_{a^++1/2}\notin \im (h^+_{\underline{a^+}})$ and  $h^-_{\underline{a^+}}(v_{a^++1/2+dN})=0$, or
\item $a^+\in \tZ$, $v_{a^+-1/2}\notin \im (h^-_{\underline{a^+}})$ and  $h^+_{\underline{a^+}}(v_{a^+-1/2+dN})=0$, or
\item $b\in \tZ$, $v_{b-1/2}\notin \im (h^-_{\underline{b}})$ and  $h^+_{\underline{b}}(v_{b-1/2+dN})=0$,
\end{itemize}
and otherwise $b\in H$. 
Note that $J'_0$ and $J''_0$ are divided into themselves. 

Let $H$ be one of the subsets of $J_d$ ($d\neq 0$).  
\end{NB}

\begin{lem}
\begin{itemize}
\item[(1)]
If $|J_0|$ is even, then the subspace of $\ker(d_1)$ spanned by $\{\beta_{a,d}\}_{a\in J_0}$ is contained in $\im(d_2)$.
If $|J_0|$ is odd, then 
\[
\dim (\{\beta_{a,d}\}_{a\in J_0}/\im(d_2)\cap\{\beta_{a,d}\}_{a\in J_0})=1.
\]
\item[(2)] 
The subspace of $\ker(d_1)$ spanned by $\{\beta_{a,d}\}_{a\in J_d}$ is contained in $\im(d_2)$.
\end{itemize}
\end{lem}
\begin{proof}
To make the notations less complicated, we give a proof only in the case of $|\hat{I}_r|=0$. 
(The argument, with a slight modification, works in the general case as well.) 
We define $C(j)$ ($n_0<j<n_1$) as in Example \ref{ex_1}.

For $d\in\Z$ and $j\in\tZ$ such that both  $j$ and $j-\h+dN$ (resp. $j+\h+dN$) are in the interval $[n_0,n_1]$ we define an element
\[
\gamma^\pm_{j,d}\in \Hom\left(C_{\mr{out}(h^\mp_{\uj})},C_{\mr{in}(h^\mp_{\uj})}\right)
\]
by 
\[
\gamma^\pm_{j,d}(v_{n})=\delta_{n,j\pm\h}v_{j\mp\h+dN}.
\]

For $d=0$, we can verify\footnote{Recall that $d_2$ is the linear combination of $\partial_c$'s (see \eqref{eq_del_c}) for equvalent classes of the cycles $c$ which appear in the superpotential. 
For example, if we put 
\[
c_{k}=\Bigl[h^+_{k+\h}\circ h^+_{k-\h}\circ h^-_{k-\h}\circ h^-_{k+\h}\Bigr]\quad (k\in \tI)
\]
then we have 
\[
\partial_{c_k}(\gamma^+_{\uj,d})=
\gamma^+_{\j,d}\circ h^-_{\uj}\circ h^-_{\uj+1}
+h^+_{\uj+1}\circ \gamma^+_{\j,d}\circ h^-_{\uj}
+h^-_{\uj+1}\circ h^+_{\uj+1}\circ \gamma^+_{\j,d}.
\]
Moreover, the third term equals to zero.}
\[
d_2(\gamma^{C(j)}_{j,0})=\beta_{j-1,0}+\beta_{j+1,0}, \quad d_2(\gamma^{-C(j)}_{j,0})=0
\]
where we put $\beta_{n_0-1/2,0}=\beta_{n_0+1/2,0}=0$.
This is followed by the first claim.

For $d\neq 0$, if $C(j)=\pm$ and $C(j+dN)=\mp$ then we can verify
\[
d_2(\gamma^+_{j,d})=d_2(\gamma^-_{j,d})=\beta_{j\mp 1/2,d}.
\]
Note that, by Lemma \ref{lem_2.25}, if we have $j$ and $d$ such that $C(j)=-$ and $C(j+dN)=+$, then we do not have $j'$ such that $C(j')=+$ and $C(j'+dN)=-$.
Then we can verify the second claim.
\begin{NB}
For $d\neq 0$, we define 
\[
J^{\mp}_d:=\{j\in J_d\mid j\pm 1/2\equiv 0\mod 2\}.
\]
Using Lemma $\ref{lem_2.25}$, 
we can verify that $\beta_{j,d}(d\neq 0, j\in J^{\pm}_d)$ is contained in the subspace spanned by 
\[
\{d_2(\alpha^+_{j,d}),d_2(\alpha^-_{j,d})\mid j\in J^{\mp}_d\}.
\]
(see Example \ref{ex_2}.)
This implies the second one.
Take $j\in J_d$ such that 
\[
h^\pm(v_{j\mp 1/2})=v_{j\pm 1/2},\quad 
h^\pm(v_{j\mp 1/2+cN})=v_{j\pm 1/2+cN},
\]
then we have 
\[
d_2(\alpha_{j,d}^\pm)=\begin{cases}
\beta_{j-1,d} &j+1\notin J_d,\\
\beta_{j+1,d} &j-1\notin J_d,\\
\beta_{j-1,d}+\beta_{j+1,d} & \text{otherwise}.
\end{cases}
\]
Take $j\in J_d$ such that 
\[
h^+(v_{j- 1/2})=0,\quad 
h^-(v_{j+ 1/2+cN})=v_{j- 1/2+cN},
\]
then we have 
\[
d_2(\alpha_{j,d}^+)=\beta_{j-1,d}.
\]

we define an element $\gamma_{j,d}$ by
\begin{itemize}
\item $\bal^+_{j,d}\in \Hom\left(C_{\mr{out}(h^+_{\underline{a}})},C_{\mr{in}(h^+_{\underline{a}})}\right)$ if $a\in\tZ$ and 
$h^+_{\underline{a}}(v_{a-1/2})=h^+_{\underline{a}}(v_{a+dN-1/2})=0$,
\item $\bal^-_{a,d}\in \Hom\left(C_{\mr{out}(h^-_{\underline{a}})},C_{\mr{in}(h^-_{\underline{a}})}\right)$ if $a\in\tZ$ and 
$h^-_{\underline{a}}(v_{c+1/2})=h^-_{\underline{a}}(v_{a+dN+1/2})=0$,
\end{itemize}
where $\bal^\pm_{j,d}$ is given by 
\[
\bal^\pm_{a,d}(v_{n})=\delta_{n,a\mp\h}v_{j+dN\pm\h}.
\]
For $a\in J_d$ such that $a^+$, $a^-\in J_d$ we have
\[
d_2(\gamma_{a,d})=\beta_{a^-,d}+\beta_{a^+,d}.
\]
Let $a_{\mr{min}}$ and $a_{\mr{max}}$ be the maximal and minimal elements in $H$.
We can verify that either $\beta_{a_{\mr{min}},d}\in \im(d_2)$ or $\beta_{a_{\mr{max}},d}\in \im(d_2)$ holds by case-by-case argument.
For example, if $a_{\mr{min}}\in\Z$ and $h^-_{\underline{{a_{\mr{min}}}^-}}(v_{{a_{\mr{min}}}^-+1/2})=v_{{a_{\mr{min}}}^--1/2}$ then we have
\[
d_2(\bal^+_{{a_{\mr{min}}}^-,d})=\beta_{a_{\mr{min}},d}. 
\]
\end{NB}
\end{proof}
\begin{NB}
\begin{figure}[htbp]
  \centering
  \input{pic10.tpc}
  \caption{example: $d_2(\alpha^+_{9/2,-1})=\beta_{11/2,-1}$.}
  \label{fig:10}
\end{figure}
\end{NB}
\begin{NB}
\begin{prop}
If $|J_0|$ is even, then $J_0$ is contained in $\im(d_2)$. 
If $|J_0|$ is odd, then we have $\dim(J_0/\im(d_2)\cap J_0)=1$. 
\end{prop}
\begin{proof}
Let $a_{\mr{min}}$ and $a_{\mr{max}}$ be the maximal and minimal elements in $J_0$. 
For $a\in J_0$ we have
\[
d_2(\gamma_{a,0})=\begin{cases}
\beta_{a_{\mr{min}}^+,d} & \text{if $a=a_{\mr{min}}$},\\
\beta_{a_{\mr{max}}^-,d} & \text{if $a=a_{\mr{max}}$},\\
\beta_{a^-,d}+\beta_{a^+,d} & \text{otherwise}.
\end{cases}
\]
Thus the claim follows.
\end{proof}
\end{NB}
The lemma above implies Proposition \ref{prop3.12.} (1).
Moreover, we can take $\alpha_{n_0+1/2,0}^{C(n_0+1/2)}$ as a generator of $\mr{Ext}^1(C,C)$ in the case $\mr{ext}^1(C,C)=1$ (we keep using the notations in the case of $|\hat{I}_r|=0$ for brevity).

For $1\leq M\leq m$, let $v_{a}^{(M)}$ denote the element in the $M$-th copy of $C$. 
By putting 
\[
h_{n_0+1/2,0}^{+}(v_{n_0}^{(M)})=
v_{n_0+1}^{(M)}+v_{n_0+1}^{(M+1)}
\]
if $C(n_0+1/2)=+$, or 
\[
h_{n_0+1/2,0}^{-}(v_{n_0+1}^{(M)})=
v_{n_0}^{(M)}+v_{n_0}^{(M+1)}
\]
if $C(n_0+1/2)=-$, we can define a $A$-module $C_m$ which is described as $m-1$ times successive extensions of $C$'s.
\begin{lem}
$\mr{hom}(C,C_m)=1$.
\begin{NB}
\begin{itemize}
\item[(1)]
The $A$-module $C_m$ is the unique $A$-module which is described as $m-1$ times successive extensions of $C$'s. 
\item[(2)] $\mr{ext}^1(C_m,C)=1$.
\end{itemize}
\end{NB}
\end{lem}
\begin{proof}
We will prove by induction.
Assume that $\mr{hom}(C,C_{m})=1$.
We apply $\Hom(C,-)$ for the short exact sequence
\[
0\to C\to C_{m+1}\to C_m\to 0
\]
to get the long exact sequence.
By the definition of $C_{m+1}$, the connected homomorphism 
\[
\Hom(C,C_m)\to \mr{Ext}^1(C,C)
\]
is isomorphism.
Hence we have $\mr{hom}(C,C_{m+1})=\mr{hom}(C,C)=1$.
\end{proof}
\begin{lem}
$\mr{ext}^1(C,C_m)=1$
\end{lem}
\begin{proof}
The proposition above implies that $C_{m+1}$ is a nontrivial extension of $C$ and $C_m$ and that $\mr{ext}^1(C,C_m)\geq 1$. 
On the other hand, the long exact sequence above implies $\mr{ext}^1(C,C_m)\leq 1$.
Then the claim follows.
\end{proof}
These lemmas above imply Proposition \ref{prop3.12.} (2).

\section{Mutations and stabilities}\label{mutation}
In the final section, we provide an alternative description of the moduli spaces for a generic stability parameter.
As by-products, we see that for a generic stability parameter "between DT and NCDT" the set of torus fixed points on the moduli space is isolated and that the weighted Euler characteristics coincide with the Euler characteristics up to signs. 
(The combinatorial description of the fixed point set is given in \cite{open_3tcy, NCDTviaVO}.)

\subsection{mutations}
In \S \ref{as_moduli} we associate a quiver with a superpotential $A:=A_\tau$ with a map $\tau\colon \hat{I}\to \{\pm 1\}$, where we use $1$ and $-1$ instead of $H$ and $S$ respectively.
For $k\in\hat{I}$, let $\mu_k(\tau)\colon \hat{I}\to \{\pm 1\}$ be the map given by 
\[
\mu_k(\tau)(l)=
\begin{cases}
\tau(k)\tau(l) & (l=k\pm 1),\\
\tau(l)& (\text{otherwise}),
\end{cases}
\]
and let $\mu_k(A)$ denote the quiver with the superpotential $A_{\mu_k(\tau)}$.

Let $P_k$ be the projective $A$-module associated with the vertex $k\in\hat{I}$ and we set $P:=\bigoplus_k P_k(=A)$.
An element in $P_k$ is a linear combination of paths which start at the vertex $k$.
We define the new $A$-module 
\begin{equation}\label{eq_P'}
P_k':=\mathrm{coker}(P_k\to P_{k-1}\oplus P_{k+1})
\end{equation}
where $P_k\to P_{k\pm 1}$ be the maps given by composing the arrows from $k\pm 1$ to $k$ to the paths. 
\begin{prop}\label{prop-mutation}
\begin{enumerate}
\item[(1)] The object $\mu_k(P)=\bigoplus_{l\neq k} P_l\oplus P_k'$ is a tilting generator in $D^b(A\text{-}\mathrm{mod})$. 
In particular, non T-invariant
\[
\mb{R}\Hom(\mu_k(P),-)=:\Psi_k\colon D^b(A\text{-}\mathrm{mod})\to D^b(\mathrm{End}_A(\mu_k(P))^{\mr{op}}\text{-}\mathrm{mod})
\]
gives an equivalence.
\item[(2)] $\mu_k(A)\overset{\sim}{\longrightarrow}\mathrm{End}_A(\mu_k(P))^{\mr{op}}$.
\end{enumerate}
\end{prop}
\begin{proof}
(1) It is clear that $\mu_k(P)$ is a generator. 

Applying the functor $\Hom(-,P_l)$ ($l\neq k$) for the complex in the definition \eqref{eq_P'}, we get
\begin{equation}\label{eq_surj}
\Hom (P_{k-1},P_l)\oplus \Hom (P_{k+1},P_l) \to \Hom (P_{k},P_l).
\end{equation}
Take a path $\mca{P}$ from $l$ to $k$.
If $k\in \hat{I}_r$ then either 
\begin{itemize}
\item $\mca{P}=(r_k)^a \circ h^+_{k-\frac{1}{2}} \circ \mca{P}'$ for a path $\mca{P}'$ from $l$ to $k-1$ and for a non-negative integer $a$, or
\item $\mca{P}=(r_k)^a \circ h^-_{k+\frac{1}{2}} \circ \mca{P}'$ for a path $\mca{P}'$ from $l$ to $k+1$ and for a non-negative integer $a$.
\end{itemize}
In each case, we have $\mca{P}=h^\pm_{k\mp\frac{1}{2}}\circ \mca{P}''$ for a path $\mca{P}''$ from $l$ to $k\pm 1$ (see the relation given in \S \ref{as_moduli}).
Thus the map in \eqref{eq_surj} is surjective and we have
\[
\mr{Ext}^i(P'_k,P_l)=0\quad (i\neq 0).
\] 
We can show the vanishing for $k\in \hat{I}_r$ and the vanishing of $\mr{Ext}^i(P_l,P'_k)$ in the same way.

The cohomologies of the total complex of the following double complex give $\mr{Ext}^i(P'_k,P'_k)$: 
\[
\xymatrix{
\Hom(P_{k-1}\oplus P_{k+1},P_k) \ar[d] \ar[r] & \Hom(P_k,P_k) \ar[d]\\
\Hom(P_{k-1}\oplus P_{k+1},P_{k-1}\oplus P_{k+1}) \ar[r]& \Hom(P_k,P_{k-1}\oplus P_{k+1})
}
\]
(the middle cohomology gives $\mr{Ext}^0(P'_k,P'_k)$).
We can verify that the left map in the total complex is injective and the right one is surjective.

Hence we have $\mr{Ext}^i(\mu_k(P),\mu_k(P))=0$ ($i\neq 0$). 
So $\mu_k(P)$ is a tilting generator, which gives an equivalence.

\smallskip
\noindent
(2) We will construct the isomorphism explicitly.
Let $H^\pm_i$ and $R_l$ denote the arrows of the mutated quiver $\mu_k(A)$, where $h^\pm_i$ and $r_l$ are the arrows of the original quiver $A$.

For an arrow in the mutated quiver from $s\neq k$ to $t\neq k$, we associate the element in $\Hom(P'_s,P'_t)=\Hom(P_s,P_t)$ which is given by the same arrow in the original quiver.

Assume that $k\in\hat{I}_r$ (we can verify for $k\notin\hat{I}_r$ in the same way).

Note that the composition of $P_k\to P_{k-1}\oplus P_{k+1}$ and 
\[
(h^\pm_{k\pm\frac{1}{2}} \circ h^\pm_{k\mp\frac{1}{2}})+(h^\pm_{k\mp\frac{1}{2}} \circ h^\mp_{k\pm\frac{1}{2}}) \colon P_{k-1}\oplus P_{k+1}\to P_{k\pm 1}
\]
is zero. 
We associate the induced map $P_k'\to P_{k\pm 1}$ with the arrow $H^\pm_{k\pm\frac{1}{2}}$. 

For the arrow $H^\mp_{k\pm\frac{1}{2}}$, we associate the map $P_{k\pm 1}\to P'_{k}$ induced by 
\[
0\oplus \mathrm{id}_{P_{k\pm 1}}\colon P_{k\pm 1}\to P_{k-1}\oplus P_{k+1}.
\]
Furthermore, the following diagram commutes:
\[
\xymatrix{
P_k \ar[d]_{r_k} \ar[r] & P_{k-1}\oplus P_{k+1} \ar[d]^{\footnotesize \begin{pmatrix}r_{k-1}& 0\\0&r_{k+1}\end{pmatrix}} \\
P_k \ar[r]& P_{k-1}\oplus P_{k+1},
}\]
where 
\[
R_{k\pm 1}=
\begin{cases}
r_{k\pm 1} & (k\pm 1\in \hat{I}_r),\\
h^\mp_{k\pm \frac{3}{2}}\circ h^\pm_{k\pm \frac{3}{2}}& (k\pm 1\notin \hat{I}_r).\\
\end{cases}
\]
We associate this with the arrow $R_k$. 

These maps satisfy the relations of the mutated quiver. 
We can also verify that the correspondence above gives an isomorphism. 
For example, the cokernel of the map \eqref{eq_surj} is isomorphic to $\Hom(P_k,P_l)$ and so we have $\Hom(P'_k,P_l)=\Hom(P_k,P_l)$. 
\end{proof}

\subsection{the affine Weyl group action}\label{subsec-weyl}
For $k\in\hat{I}$ we define the map $\mu_k\colon \Z^{\hat{I}}\to\Z^{\hat{I}}$ by 
\[
(\mu_k(\mathbf{v}))_l=
\begin{cases}
\mathbf{v}_{k-1}-\mathbf{v}_k+\mathbf{v}_{k+1} & l=k,\\
\mathbf{v}_{l} & \text{otherwise}
\end{cases}
\]
for $\mathbf{v}\in\Z^{\hat{I}}$.
We also define $\mu_k\colon \R^{\hat{I}}\to\R^{\hat{I}}$ by
\[
(\mu_k(\zeta))_l=
\begin{cases}
-\zeta_k & l=k,\\
\zeta_{k\pm 1}+\zeta_k & l=k\pm 1,\\
\zeta_{l} & \text{otherwise}
\end{cases}
\]
for $\zeta\in\R^{\hat{I}}$.
Note that 
\[
\mathbf{v}\cdot \zeta =\mu_k(\mathbf{v})\cdot\mu_k(\zeta)
\]
for any $\mathbf{v}$ and $\zeta$.

\begin{NB}
\begin{prop}
Let $\alpha\in \Z^{\hat{I}}$ be a positive root and $\zeta$ be a parameter such that $\alpha\cdot \zeta=0$, $\zeta_k<0$ and such that $\zeta$ is not on any other wall.
Given a $\zeta$-stable $A$-module $C$ with $\dimv\,{C}=\alpha$ and $\mathcal{I}\cdot C=0$ for a maximal ideal $\mathcal{I}\subset R$.
Then $\Psi_k(C)$ is the unique $\mu_k(\zeta)$-stable $\mu_k(A)$-module with $\dimv(\Psi_k(C))=\mu_k(\alpha)$ such that $\mathcal{I}\cdot \Psi_k(C)=0$.
\end{prop}
\begin{proof}
Since $C$ is $\zeta$-stable and $\zeta_k<0$, the map
\[
C_{k-1}\oplus C_{k+1}\to C_k
\]
is surjective. 
Hence $\Psi_k(C)$ is a $\mu_k(A)$-module.
Suppose we have an exact sequence 
\[
0\to V\to \Psi_k(C)\to V'\to 0
\]
of nontrivial $\mu_k(A)$-modules. 
Let $H^*_{A}(-)$ denote the cohomology with respect to the natural t-structure of $D^b(A\text{-}\mathrm{mod})$.
Note that for a $\mu_k(A)$-module $V$, $H^*_{A}(\Psi^{-1}_k(V))$ is concentrated on the degree $0$ and $-1$ and $(H^{-1}_{A}(\Psi^{-1}_k(V)))_l=0$ for $l\neq k$. 
By the long exact sequence, we have $H^{-1}_{A}(\Psi^{-1}_k(V))=0$ and the following exact sequence
\[
0\to H^{-1}_{A}(\Psi^{-1}_k(V'))\to \Psi^{-1}_k(V) \to C\to H^{0}_{A}(\Psi^{-1}_k(V'))\to 0
\]
of $A$-modules.
Then we have
\begin{align*}
\mu_k(\zeta)\cdot \dimv\,V&=\zeta\cdot\dimv\,\Psi^{-1}_k(V)\\
&<\zeta\cdot \dimv\,H^{-1}_{A}(\Psi^{-1}_k(V'))\\
&\leq 0
\end{align*}
where the first inequality follows from the $\zeta$-stablility of $C$ and the second one follows from the assumption $\zeta_k<0$.
\end{proof}
\end{NB}

In the rest of this section, we fix a parameter $\zeta$ such that $\sum \zeta_l<0$ and such that $\zeta+c\cdot \mathbf{1}$ is not on an intersection of two walls for any $c\in\R$. See Remark \ref{rem-+} for the case $\sum \zeta_l>0$.
Then we get the sequence $H_0,\ldots,H_r$ of chambers such that
\begin{itemize}
\item $\zeta-c\cdot \mathbf{1}\in\cup\,\overline{H_s}$ for any $c\geq 0$,
\item for any $H_s$, there exists some $c\geq 0$ such that $\zeta-c\cdot \mathbf{1}\in {H_s}$, and 
\item suppose $\zeta-c\cdot \mathbf{1}\in {H_s}$, $\zeta-c'\cdot \mathbf{1}\in {H_{s'}}$ and $s<s'$, then $c>c'$.
\end{itemize}
(See Figure \ref{fig:C} for example.)
\begin{figure}[htbp]
  \centering
  \input{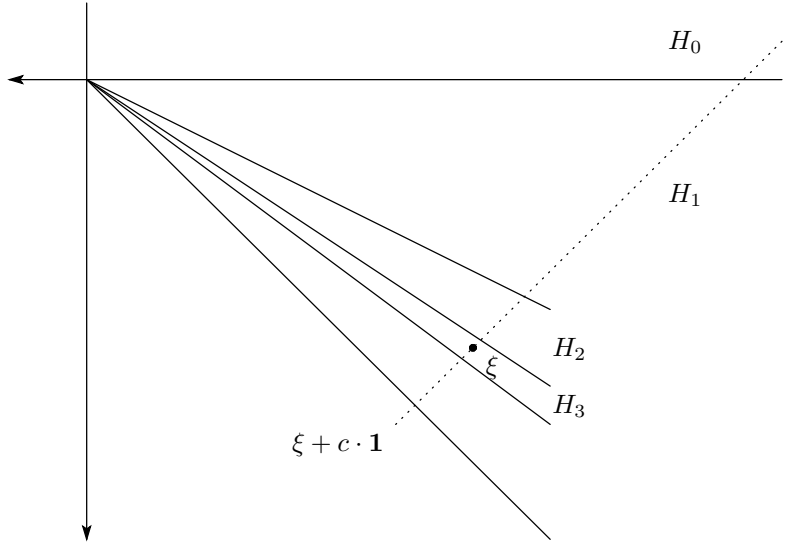}
  \caption{example of $H_s$}
  \label{fig:C}
\end{figure}
The chamber $H_0$ is given by $\{\xi\mid \xi\cdot \alpha_k<0\}$, 
and its boundary is the union of subsets of the hyperplanes
\[
W_\alpha=\{\xi\mid \xi\cdot \alpha_k=0\}.
\]
Take $k_1$ such that $\overline{H_{0}}\cap\overline{H_{1}}$ is in $W_{\alpha_{k_1}}$.
Then, the chamber $H_1$ is given by
\[
\{\xi\mid \xi\cdot \mu_{k_1}(\alpha_k)<0\}
\]
and its boundary is the union of subsets of the hyperplanes
\[
W_{\mu_{k_1}(\alpha_k)}=\{\xi\mid \xi\cdot {\mu_{k_1}(\alpha_k)}=0\}.
\]
In other words, we have $H_1=\mu_{k_1}(H_0)$.
We can repeat to get the sequence $k_1,\ldots, k_r$ such that 
\[
H_s=\mu_{k_s}\circ\cdot\circ\mu_{k_1}(H_0)\quad (1\leq s\leq r).
\]
We denote $\mu_s:=\mu_{k_{s}}\circ\cdots\circ\mu_{k_{1}}$ and $\mu_\zeta:=\mu_r$.
We define $\Psi_s$ as the composition
\[
D^b(A\text{-}\mathrm{mod})
\underset{\Psi_{k_1}}{\overset{\sim}{\longrightarrow}}
D^b(\mu_1(A)\text{-}\mathrm{mod})
\underset{\Psi_{k_2}}{\overset{\sim}{\longrightarrow}}
\cdots
\underset{\Psi_{k_r}}{\overset{\sim}{\longrightarrow}}
D^b(\mu_r(A)\text{-}\mathrm{mod})
\]
and we put $\Psi_\zeta:=\Psi_r$.
\begin{prop}\label{prop3.2}
Let $C$ be a $\theta_{(\mu_{s})^{-1}(\xi)}$-stable $\mu_s(A)$-module for $\xi\in H_{s+1}$.
If $C$ is the simple $\mu_s(A)$-module supported on the vertex $k_{s+1}$, then $\Phi_{k_{s+1}}(C)$ is the simple $\mu_{s+1}(A)$-module supported on the vertex $k_{s+1}$ shifted by one.
Otherwise $\Phi_{k_{s+1}}(C)$ is $\theta_{(\mu_{s+1})^{-1}(\xi)}$-stable $\mu_{s+1}(A)$-module.
\end{prop}
\begin{proof}
The claim in the second sentence is trivial.
For the third sentence, it is enough to show in the case $s=0$.

Since $\theta_{\xi}$-stability is equivalent to $\theta_{\xi'}$-stability ($\xi'=\xi+c\cdot \mathbf{1}$) for any $c$, we may assume $\underline{\dim}(C)\cdot\xi=0$. Then $\xi$ in not in $H_1$ any more but we have $\xi_{k_1}<0$.

Since $C$ is $\theta_{\xi}$-stable for $\xi_{k_1}<0$, the map
\[
C_{k_1-1}\oplus C_{k_1+1}\to C_{k_1}
\]
is surjective. 
Hence $\Psi_{k_1}(C)$ is a $\mu_{k_1}(A)$-module.
Suppose we have an exact sequence 
\[
0\to V\to \Psi_{k_1}(C)\to V'\to 0
\]
of nontrivial $\mu_{k_1}(A)$-modules. 
Let $H^*_{A}(-)$ denote the cohomology with respect to the natural t-structure of $D^b(A\text{-}\mathrm{mod})$.
Note that for a $\mu_{k_1}(A)$-module $V$, $H^*_{A}(\Psi^{-1}_{k_1}(V))$ is concentrated on the degree $0$ and $-1$ and $(H^{-1}_{A}(\Psi^{-1}_{k_1}(V)))_l=0$ for $l\neq k_1$. 
In particular, we have 
\[
\xi\cdot \underline{\dim}H^{-1}_{A}(\Psi^{-1}_{k_1}(V))\leq 0.
\] 
By the long exact sequence, we have $H^{-1}_{A}(\Psi^{-1}_{k_1}(V))=0$ and the following exact sequence
\[
0\to H^{-1}_{A}(\Psi^{-1}_{k_1}(V'))\to \Psi^{-1}_{k_1}(V) \to C\to H^{0}_{A}(\Psi^{-1}_{k_1}(V'))\to 0
\]
of $A$-modules.
Since $C$ is $\theta_\xi$-stable and $\xi\cdot\underline{\dim}(C)=0$, we have 
\[
\xi\cdot\Bigl(\underline{\dim}\,\Psi^{-1}_{k_1}(V)-\underline{\dim}\,H^{-1}_{A}(\Psi^{-1}_{k_1}(V'))\Bigr)<0.
\] 
Then we have
\[
\mu_{k_1}(\xi)\cdot \dimv\,V=\xi\cdot\dimv\,\Psi^{-1}_{k_1}(V)<\xi\cdot \dimv\,H^{-1}_{A}(\Psi^{-1}_{k_1}(V'))\leq 0.
\]
Hence $\Psi_{k_1}(C)$ is $\mu_{k_1}(\xi)$-stable.
\end{proof}
\begin{NB}
\begin{prop}\label{cor-mutation-stable}
\begin{enumerate}
\item[(1)] Let $C$ be a finite dimensional $\theta_\zeta$-stable $A$-module.
Suppose $\dimv(C)\cdot\zeta <0$, then $\Psi_\zeta(C)\in(\mu_\zeta(A))\text{-}\mathrm{mod}$.
Suppose $\dimv(C)\cdot\zeta >0$, then $\Psi_\zeta(C)\in(\mu_\zeta(A))\text{-}\mathrm{mod}[1]$.
\item[(2)] Let $C'$ be a finite dimensional $\theta_{\mu_{\zeta}(\zeta)}$-stable $\mu_\zeta(A)$-module.
Take $d\in\R$ such that $(\mu_{\zeta}(\zeta)+d\cdot \mathbf{1})\in \mu_\zeta(C^+)$ where $C^+:=\{\zeta\mid \zeta_l>0\}$.
Suppose $\dimv(C')\cdot(\mu_{\zeta}(\zeta)+d\cdot \mathbf{1}) >0$, then $\Psi^{-1}_\zeta(C')\in A\text{-}\mathrm{mod}$.
Suppose $\dimv(C')\cdot(\mu_{\zeta}(\zeta)+d\cdot \mathbf{1}) <0$, then $\Psi^{-1}_\zeta(C')\in A\text{-}\mathrm{mod}[-1]$.
\end{enumerate}
\end{prop}
\begin{proof}
If 
\end{proof}
\end{NB}
\begin{cor}\label{cor3.3}
Let $C$ be a $\theta_\zeta$-stable $A$-module.
If $\zeta\cdot \underline{\dim}(C)<0$, then $\Phi_\zeta(C)\in \amod$.
If $\zeta\cdot \underline{\dim}(C)>0$, then $\Phi_\zeta(C)\in \amod[1]$.
\end{cor}
\begin{proof}
Note that there exists $1\leq s\leq r$ such that $\alpha^s=\underline{\dim}(C)$ if and only if $\zeta\cdot \underline{\dim}(C)>0$.
By Proposition \ref{prop3.2}, we can verify that if there exists $1\leq s\leq r$ such that $\alpha^s=\underline{\dim}(C)$ then $\Phi_\zeta(C)\in \amod[1]$ and otherwise $\Phi_\zeta(C)\in \amod$.
\end{proof}
\begin{rem}
A more general statement is given in \cite[Lemma 2.4]{cluster-via-DT}.
\end{rem}

\subsection{the tilting generator is a vector bundle}\label{subsec-isolated}
We identify $D^b(\coh)$ and $D^b(\mr{mod}A)$ by the equivalence constructed in \S \ref{derived-equiv} (here we use the notation without the subscripts $\sigma$ as we mentioned at the beginning of \S \ref{counting}). 
In this subsection, we assume that $N\cdot\zeta_k<\sum \zeta_l$ for any $k\neq 0$\footnote{This condition determines a chamber on the hyperplane $\{\zeta'\mid \sum \zeta'_i=0\}$ in which the line $\zeta+c\cdot \mathbf{1}$ intersect with the hyperplane. A chamber in this hyperplane determines a crepant resolution. If we take another chamber, Proposition \ref{prop-lb} holds for the corresponding crepant resolution.}.

$\iota_k\colon \tilde{I}\to \tilde{I}$ be the permutation of $k\pm 1/2$.
We put $\iota_s=\iota_{k_1}\circ\cdots\circ \iota_{k_s}$ and $\iota_\zeta=\iota_r$.
\begin{prop}\label{prop-lb}
Let $P_{\zeta,k}$ be the projective $\mu_\zeta(A)$-module.
Then $\Psi_\zeta^{-1}(P_{\zeta,k})$ is a line bundle $\mathcal{L}^\zeta_k$ on $\mca{Y}$.
Moreover, it is the following map that is associated with the arrow $H^\pm_i$:
\[
\mathcal{L}^\zeta_{i\mp\frac{1}{2}}\hookrightarrow \mathcal{L}^\zeta_{i\mp\frac{1}{2}}\otimes \OO_Y(E_{\iota_\zeta(i)})\simeq\mathcal{L}^\zeta_{i\pm\frac{1}{2}}.
\]
\end{prop}
\begin{proof}
We prove the claim by induction with respect to $r$.

By the argument in the proof of Proposition \ref{smallseq}, we can verify that the sequence
\[
0\to \OO\to \OO(E^+_i)\oplus \OO(E^-_j)\to \OO(E^+_i+E^-_j)\to 0
\]
is exact if $i<j$.
Thus we have
\[
(\mu_s(P))_{k_s}\simeq {P}_{k_s}\otimes \OO(E^+_{\iota_s(k_s-1/2)}+E^-_{\iota_s(k_s+1/2)}).
\]
The second half of the claim can be verified by the explicit description of the endomorphism algebra in the proof of Proposition \ref{prop-mutation}.
\end{proof}

\begin{NB}
In this subsection, we assume that $N\cdot\zeta_k<\sum \zeta_l$ for any $k\neq 0$.
Let $\nu_i$ ($i\in\tilde{I}$) be real numbers such that
\[
\nu_{k+\frac{1}{2}}-\nu_{k-\frac{1}{2}}
=
N\cdot\zeta_k-\sum \zeta_l
\]
for any $k\in\hat{I}$.
We denote by $\pi_k\colon \R^{\tilde{I}}\to\R^{\tilde{I}}$ the permutation of the $(k-1/2)$-th and $(k-1/2)$-th elements and set $\pi_s:=\pi_{k_s}\circ\cdots\circ\pi_{k_1}$.
Note that 
\[
(\pi_s(\nu))_{k+\frac{1}{2}}-(\pi_s(\nu))_{k-\frac{1}{2}}
=
N\cdot(\mu_s(\zeta))_k-\sum \zeta_l.
\]
By the definition of the sequence $(k_s)$, we have $(\mu_{s-1}(\zeta))_{k_s}>0$.
Recall that we assume $\sum\zeta_l<0$. Thus we have
\begin{equation}\label{eq->}
(\pi_{s-1}(\nu))_{k_s+\frac{1}{2}}<(\pi_{s-1}(\nu))_{k_s-\frac{1}{2}}.
\end{equation}
Let $\iota_s\colon \tilde{I}\to \tilde{I}$ be the permutation such that the sequence $((\pi_{s}(\nu))_{\iota_s(i)})_{i\in\tilde{I}}$ is increasing and set $\iota:=\iota_r$.

\begin{prop}\label{prop-lb}
Under the derived equivalence given in \S \ref{derived-equiv}, the direct summand $(\mu_\zeta(P))_k$ is a line bundle $\mathcal{L}^\zeta_k$ on $Y$. 
Moreover, it is the following map that is associated with the arrow $H^\pm_i$:
\[
\mathcal{L}^\zeta_{i\mp\frac{1}{2}}\hookrightarrow \mathcal{L}^\zeta_{i\mp\frac{1}{2}}\otimes \OO_Y(E_{\iota(i)})\simeq\mathcal{L}^\zeta_{i\pm\frac{1}{2}}.
\]
\end{prop}
\begin{proof}
We prove the claim by induction.

By the argument in the proof of Proposition \ref{smallseq}, we can verify that the sequence
\[
0\to \OO\to \OO(E^+_i)\oplus \OO(E^-_j)\to \OO(E^+_i+E^-_j)\to 0
\]
is exact if $i<j$.
Thus we have
\[
(\mu_s(P))_{k_s}\simeq {P}_{k_s}\otimes \OO(E^+_{\iota_s(k_s-1/2)}+E^-_{\iota_s(k_s+1/2)}).
\]
The second half of the claim can be verified by the explicit description of the endomorphism algebra in the proof of Proposition \ref{prop-mutation}.
\end{proof}

\begin{cor}\label{cor-vanishing}
\[
H^{1}(Y,(\mathcal{L}^\zeta_k)^{-1})=0.
\]
\end{cor}
\begin{proof}
By the argument in the proof of Proposition \ref{prop-lb}, $\mathcal{L}^\zeta_k$ is associated with the divisor which is described as a sum of divisors of the form $E^+_i+E^-_j$ ($i<j$).
We can verify the claim using the description \eqref{eq-h1} of $H^1$ of a line bundle on $Y$.
\end{proof}

\begin{prop}
The set of the torus fixed points $\mathfrak{M}_\zeta(\mathbf{v})^T$ is isolated.
\end{prop}
\begin{proof}
By Corollary \ref{cor-vanishing} we have $\OO_Y\in \mathcal{P}_\zeta$. 
Let $P_\zeta$ be the $\mu_k(A)$-module corresponding to $\OO_Y$. 
By Proposition \ref{prop-4.3} and Proposition \ref{prop-4.6}, the moduli space $\mathfrak{M}_\zeta(\mathbf{v})$ parametrizes finite dimensional quotient $\mu_k(A)$-modules $V'$ of $P_\zeta$ with $\dimv\,V'=\mu_k(\mathbf{v})$.
Note that
\[
(P_\zeta)_k=H^0(Y,(\mathcal{L}^\zeta_k)^{-1})
\]
and the $T$-weight decomposition of $H^0$ of a line bundle on $Y$ is multiplicity free.
Hence the claim follows.
\end{proof}
\end{NB}
As a corollary of the proposition, we can see that the set of the torus fixed points $\mathfrak{M}_\zeta(\mathbf{v})^T$ is isolated. 
In \S \ref{sub_torus}, we will show a stronger result.

\subsection{mutations and change of stability parameters}
We set $\mathcal{P}:=\afmod$ and denote by $\mathcal{P}_\zeta$ the image of the Abelian category $\mu_\zeta(A)\text{-}\mathrm{fmod}$ under the equivalence $\Psi^{-1}_\zeta$. 
Let $\overline{\mathcal{P}}$ (resp. $\overline{\mathcal{P}_\zeta}$) be the category consisting of pairs $(V,W,s)$, where $V\in\mathcal{P}$ (resp, $\in\mathcal{P}_\zeta$), $W$ is a finite dimensional vector space and $s\colon P_0\otimes W\to V$. An object $(V,W,s)$ with $1$-dimensional $W$ is simply written $(F,s)$.
Note that $\overline{\mathcal{P}}$ is equivalent to $\hafmod$.

\begin{defn}
For $\zeta\in\R^{\hat{I}}$, we say an object $(V,s)\in\mathcal{P}_\zeta$ is $(\zeta,\mathcal{P}_\zeta)$-(semi)stable if the following conditions are satisfied:
\begin{enumerate}
\item[(A)] for any nonzero subobject $0\neq S\subseteq V$ in $\mathcal{P}_\zeta$, we have
\[
\zeta\cdot\dimv\,S\,(\le)\,0,
\]
\item[(B)] for any proper subobject $T\subsetneq V$ in $\mathcal{P}_\zeta$ which $s$ factors through, we have
\[
\zeta\cdot\dimv\,T\,(\le)\,\zeta\cdot\dimv\,V.
\]
\end{enumerate}
\end{defn}
From now on, the $\zeta$-(semi)stability for objects in $\overline{\mathcal{P}}\simeq\hafmod$ is written as the "$(\zeta,\overline{\mathcal{P}})$-(semi)stability". 
We set $\zeta_{\mathrm{cyclic}}:=\mu_\zeta(\zeta)$. Note that $(\zeta_{\mathrm{cyclic}})_l<0$ for any $l\in\hat{I}$.
The following four claims follow from Corollary \ref{cor3.3} by the same argument as in \cite[\S 4.4]{nagao-nakajima}: 

\begin{lem}\label{lem-4.2}\textup{(see \cite[Lemma 4.8]{nagao-nakajima})}
Let $(F,s)$ be a $(\zeta,\overline{\mathcal{P}})$-stable object, then $F\in\mathcal{P}_\zeta$.
\end{lem}

\begin{prop}\label{prop-4.3}\textup{(see \cite[Proposition 4.9]{nagao-nakajima})}
Let $(F,s)$ be a $(\zeta,\overline{\mathcal{P}})$-stable object, then $(F,s)$ is $(\zeta_{\mathrm{cyclic}},\overline{\mathcal{P}_\zeta})$-stable.
\end{prop}

\begin{lem}\label{lem-4.5}\textup{(see \cite[Lemma 4.11]{nagao-nakajima})}
Let $(F,s)$ be a $(\zeta_{\mathrm{cyclic}},\overline{\mathcal{P}_\zeta})$-stable object, then $F\in\mathcal{P}$.
\end{lem}

\begin{prop}\label{prop-4.6}\textup{(see \cite[Proposition 4.12]{nagao-nakajima})}
Let $(F,s)$ be a $(\zeta_{\mathrm{cyclic}},\overline{\mathcal{P}_\zeta})$-stable object, then $(F,s)$ is $(\zeta,\overline{\mathcal{P}})$-stable.
\end{prop}
\begin{rem}\label{rem-+}
\begin{enumerate}
\item[(1)]
For a parameter $\zeta$ such that $\sum \zeta_l>0$, we can apply all the argument after a slight modification: 
first, the information in which chamber $\zeta+d\cdot \mathbf{1}$ ($d\geq 0$) is contained defines the sequence $k_1,\ldots,k_r$. 
We define $\mu_\zeta$ and $\Psi_\zeta$ in the same way and denote by $\mathcal{P}_\zeta$ the image of $(\mu_\zeta(A))\text{-}\mathrm{fmod}$ under the equivalence $\Psi^{-1}_\zeta[-1]$.
\item[(2)]
The Abelian category $A\text{-}\mathrm{fmod}$ and the function 
\[
Z_{\zeta_{\mathrm{cyclic}}}\colon K(A\text{-}\mathrm{fmod})\simeq \Z^{\hat{I}}\to \C
\]
given by 
\[
Z_{\zeta_{\mathrm{cyclic}}}(\mathbf{v}):=(\mathbf{1}\cdot \mathbf{v})+(-\zeta_{\mathrm{cyclic}}\cdot \mathbf{v})\sqrt{-1}
\]
provide a Bridgeland stability condition $(Z_{\zeta_{\mathrm{cyclic}}},\{\mathrm{P}(\phi)\}_{\phi\in\R})$, where $\mathrm{P}(\phi)$ is the full subcategory of $D^b(A\text{-}\mathrm{fmod})$ of semistable objects with phase $\phi$.
For $\phi\in\R$, let $\mathcal{P}_{[\phi,\phi+1)}$ denote the full subcategory of objects such that all the factors in their Hardar-Narashimhan filtrations are in
\[
\bigcup_{\phi\leq\phi'<\phi+1}\mathrm{P}(\phi').
\]
Then we have
\[
\mathcal{P}_{[0,1)}=A\text{-}\mathrm{fmod},\quad \mathcal{P}_{[\phi(d),\phi(d)+1)}=\mu_\zeta(A)\text{-}\mathrm{fmod}
\]
where $\phi(d)$ is determined by the condition $0\leq \phi(d)<1/2$ and $\tan(\phi(d)\pi)=d$.
This is the reason why our argument works (see \cite{open_3tcy} for the detail).
\begin{NB}
\item[(3)]
Let $A=(Q,\omega)$ be a quiver with a superpotential which is $3$-dimensional Calabi-Yau and does not have any cycles of length $1$ or $2$. 
For a vertex $k$, we can mutate $A$ at $k$ to get a new one $\mu_k(A)$ and we have a derived equivalence 
\[
\Psi_k\colon D^b(\mathrm{mod}(A))\to D^b(\mathrm{mod}(\mu_k(A))).
\]
Let $\zeta\in\R^{Q_0}$ be a parameter. 
Assume that there exist a sequence $k_1,\ldots,k_s$ satisfying the conditions in \S \ref{subsec-weyl}. 
We define the category $\overline{\mathcal{P}_\zeta}$ in the same way. 
Then the same statement as Proposition \ref{prop-4.3} and \ref{prop-4.6} hold.
\end{NB}
\end{enumerate}
\end{rem}
By the claims above, we can identify the moduli space of $\zeta$-stable framed modules over the original quiver with the one of $(\zeta_{\mathrm{cyclic}},\overline{\mathcal{P}_\zeta})$-stable objects.

\subsection{Torus fixed points and weighted Euler characteristics}\label{sub_torus}
The action of $2$-dimensional torus $T'$ (see \S \ref{as_moduli}) on the moduli spaces preserves the symmetric obstruction theory. 
We want to compute the weighted Euler characteristic $D_\zeta(\mathbf{v})$ by using Behrend-Fantechi's torus localization theorem (\cite{behrend-fantechi}). 
\begin{NB}
We can verify the following
\begin{itemize}
\item
the set of $T'$-fixed closed points $\M{\zeta}{\vv}^{T'}$ is isolated,
\item
For each $T'$-fixed closed point $P\in\M{\zeta}{\vv}^{T'}$, 
the Zariski tangent space to $\M{\zeta}{\vv}$ at $P$ has no $T'$-invariant subspace,
\item
For each $T'$-fixed point $P\in\M{\zeta}{\vv}^{T'}$, 
the parity of the dimension of the Zariski tangent space to $\M{\zeta}{\vv}$ at $P$ is same as the parity of $\sum_{k\in \hI _r}v_{k}+\sum_{k\in I}v_{k}$.
\end{itemize}
\end{NB}
To apply the localization theorem, we need the combinatorial description of the set of torus fixed points given in \cite{open_3tcy,NCDTviaVO}. 
Here, we show sketches of the proofs. 

In \cite[\S 3]{NCDTviaVO}, we define a {\it crystal} of type $(\sigma,\theta,\underline{\nu},\underline{\lambda})$ for the following data:
\begin{itemize}
\item $\sigma$ is same as this paper, 
\item $\theta$ determines the chamber in which the stability parameter $\zeta$ is,  
\item $\underline{\nu}$ and $\underline{\lambda}$ are sequences of Young diagrams which determines ``asymptotic behaviors'' of the (complexes of) sheaves on $Y_\sigma$.
\footnote{In \cite{open_3tcy,NCDTviaVO}, the author study ``open'' version of the invariants. If we take the sequences of the empty Young diagrams as $\underline{\nu}$ and $\underline{\lambda}$, then the invariants coincide with the ones which we study in this paper.}
\end{itemize}
Given a crystal $C$ of type $(\sigma,\theta,\underline{\nu},\underline{\lambda})$, we can construct a $\mu_\zeta(A)$-module $M(C)$ with the basis parametrized by $C$ \footnote{In \cite{NCDTviaVO}, we construct a $\mu_\zeta(A)$-module for a {\it transition of Young diagrams}. Giving a transition of Young diagrams is equivalent to giving a crystal as is mentioned in \cite{NCDTviaVO}.}. 
For an element $B$ in the crystal $C$, let $v_B$ denote the corresponding element in $M(C)$. 
We have a distinct crystal $C^{\sigma,\theta,\underline{\nu},\underline{\lambda}}_\mr{min}$ (which is denoted by $P(\mca{V}_{\mr{min}})$ in \cite{NCDTviaVO})\footnote{This should be called the {\it grand state} of the crystal model associated to the data $(\sigma,\theta,\underline{\nu},\underline{\lambda})$}. 
The moduli spaces $\sqcup_{\vv}\mf{M}_\zeta(\vv)$ coincide with the moduli spaces of quotient modules of $M\bigl(C^{\sigma,\theta,\underline{\emptyset},\underline{\emptyset}}_\mr{min}\bigr)$.
\begin{prop}
The set of $T'$-fixed points of $\sqcup_{\vv}\mf{M}_\zeta(\vv)$ is parameterized by the set of crystals of type $(\sigma,\theta,\underline{\emptyset},\underline{\emptyset})$.
\end{prop}
\begin{proof}
We put $C:=C^{\sigma,\theta,\underline{\emptyset},\underline{\emptyset}}_\mr{min}$ and $M:=M(C)$.
For an element $B\in C$, 
$v_B$ is a $T$-weight vector in $M$.
The $T$-weights are multiplicity free.
\begin{NB}
For a vertex $q$ in $P_\sigma$, let $\tilde{\omega}_q\in \C Q$ be the sum of all representative of $w_q\in\C Q/[\C Q,\C Q]$, i.e.
\[
\tilde{\omega}_q=e_r\cdot e_{r-1}\cdot\cdots\cdot e_1+e_1\cdot e_r\cdot\cdots\cdot e_2+\cdots + e_{r-1}\cdot\cdots\cdot e_1\cdot e_r
\]
where $e_r\cdot e_{r-1}\cdot\cdots\cdot e_1$ is a representative $w_q$.
We put $c:=\sum_q \tilde{w}_q$\footnote{We have $c=XY=Z^{N_1}W^{N_0}$ under the isomorphism $Z(A_\sigma)=\C[X,Y,Z,W]/(XY-Z^{N_1}W^{N_0})$}. 
\end{NB}
We put
\[
c:=XY=Z^{N_1}W^{N_0}\in \C[X,Y,Z,W]/(XY-Z^{N_1}W^{N_0})\simeq Z(A_\sigma).
\]
Any $T'$-weight space of $M$ is described as $\C [c]\cdot v_B$ for some $v_B$. 

Assume that $F$ is a $T'$-invariant submodule of $M$. 
Then $M/F$, as a complex of coherent sheaves on $Y_\sigma=\mr{Spec}Z(A_\sigma)$ under the derived equivalence, must be supported on the origin $0\in \mr{Spec}Z(A_\sigma)$ and so any $T'$-weight space of $F$ is described as $I\cdot v_B$ for a monomial ideal $I\subset\C [c]$. 
Thus we have a subset $C_F$ of $C^{\sigma,\theta,\underline{\emptyset},\underline{\emptyset}}_\mr{min}
$ such that $F$ is spanned by the basis elements which correspond to the elements in $C_F$.
We can verify that such a subspace is preserved by the $A_\sigma$-action if and only if $C_F$ is a crystal. 
\end{proof}
\begin{lem}\label{lem_T'inv}
Let $F$ be a $T'$-invariant submodule of $M$.
Then neither \textup{(1)} $\mr{Ext}^1_{A_\sigma}(M/F,M/F)$ nor \textup{(2)} $\mr{Hom}_{A_\sigma}(F,M/F)$ has any $T'$-invariant subspaces. 
\end{lem}
\begin{proof}
For (1), we can verify the claim in the same way as \cite[Proposition 4.22]{nagao-nakajima}.
For (2), take an element $B\in C\backslash C_F$ such that $v_B\notin A_{\sigma}^{>0}\cdot M$ where $A_{\sigma}^{>0}\subset A_\sigma$ is the subalgebra which consists of paths with positive lengths. 
Let $i$ be the vertex such that $v_B\in M_i$.
We put 
\[
x=h_{i-\h}^+\circ h_{i-\h}^-,\quad 
y=
\begin{cases}
h_{i+\h}^-\circ h_{i+\h}^+ & i\notin I_r, \\
r_i & i \in I_r.
\end{cases}
\]
Note that we have $c\cdot e_i=x\cdot y=y\cdot x$. 
We can verify the claim in the same way as \cite[Proposition 4.21 (2)]{nagao-nakajima} using $x$ and $y$ instead of $b_1a_1$ and $b_2a_2$.

\end{proof}
\begin{prop}
For each $T'$-fixed closed point $P\in\M{\zeta}{\vv}^{T'}$, 
the Zariski tangent space to $\M{\zeta}{\vv}$ at $P$ has no $T'$-invariant subspaces.
\end{prop}
\begin{proof}
The claim follows from Lemma \ref{lem_T'inv} as \cite[Proposition 4.20]{nagao-nakajima}.
\end{proof}
\begin{prop}
For each $T'$-fixed point $P\in\M{\zeta}{\vv}^{T'}$, 
the parity of the dimension of the Zariski tangent space to $\M{\zeta}{\vv}$ at $P$ is same as the parity of $\sum_{k\in \hI _r}v_{k}+\sum_{k\in I}v_{k}$.
\end{prop}
\begin{proof}
We regard $\M{\zeta}{\vv}$ as the moduli space of $\tilde{A}$-modules.
The Zariski tangent space is identified with 
$\mathrm{Ext}^1_{\tilde{A}}(P,P)_0$ where the subscript $0$ denotes the trace free part. 
The dimension of the Zariski tangent space equals to
\begin{equation}\label{eq_parity}
\mathrm{ext}^1_{\tilde{A}}(P,P)-
1
=\mathrm{ext}^1_{\tilde{A}}(P,P)-
\mathrm{hom}_{\tilde{A}}(P,P).
\end{equation}
Even though $\tilde{A}$ is not $3$-Calabi-Yau (that is, the Koszul complex is not exact), as in the proof of \cite[Theorem 7.6]{joyce-song} we can compute the value of \eqref{eq_parity} by the same complex for $\tilde{A}$ as the one induced by the Koszul complex. 
Then we can compute the parity of the value of \eqref{eq_parity} in the same way as the second half of the proof of \cite[Theorem 7.1]{ncdt-brane}.
\end{proof}
\begin{NB}
We can verify that 
\begin{itemize}
\item the set of $T'$-fixed points of $\sqcup_{\vv}\mf{M}_\zeta(\vv)$ is parameterized by the set of crystals of type $(\sigma,\theta,\underline{\emptyset},\underline{\emptyset})$ as \cite[Proposition 4.14]{nagao-nakajima},
\item
for each $T'$-fixed closed point $P\in\M{\zeta}{\vv}^{T'}$, 
the Zariski tangent space to $\M{\zeta}{\vv}$ at $P$ has no $T'$-invariant subspace as \cite[Proposition 4.20]{nagao-nakajima},
\item
For each $T'$-fixed point $P\in\M{\zeta}{\vv}^{T'}$, 
the parity of the dimension of the Zariski tangent space to $\M{\zeta}{\vv}$ at $P$ is same as the parity of $\sum_{k\in \hI _r}v_{k}+\sum_{k\in I}v_{k}$ \cite[Corollary 4.19]{nagao-nakajima}.
\end{itemize}
\end{NB}
Now we see the following:
\begin{thm}\label{thm_mr}
We have
\begin{align}\label{z}
\mca{Z}_{\zeta}(\q)&=\sum_{\vv\in(\Z_{\geq 0})^{Q_0}}(-1)^{\sum_{k\in \hI _r}v_{k}+\sum_{k\in I}v_{k}}\left|\,\M{\zeta}{\vv}^T\right|\cdot\q^\vv\\
&=\mca{Z}^{\mathrm{eu}}_{\zeta}(\mathbf{p})\notag
\end{align}
where the new formal variables $\mathbf{p}$ are given by
\[
p_k=
\begin{cases}
q_k & k\neq 0, k\in\hat{I}_r,\,\text{or},\, k=0, k\neq\hat{I}_r,\\
-q_k & k\neq 0, k\neq\hat{I}_r,\,\text{or},\, k=0, k\in\hat{I}_r.
\end{cases}
\]
\end{thm}
\begin{NB}
\begin{proof}
By the claims above, we can identify the moduli space with the one of cyclic modules over the mutated quiver.
Then we can apply \cite[Theorem 7.1]{ncdt-brane}.
\end{proof}
\subsection{the tilting generator is a vector bundle}\label{subsec-isolated}
In this subsection, we assume that $N\cdot\zeta_k<\sum \zeta_l$ for any $k\neq 0$.
Let $\nu_i$ ($i\in\tilde{I}$) be real numbers such that
\[
\nu_{k+\frac{1}{2}}-\nu_{k-\frac{1}{2}}
=
N\cdot\zeta_k-\sum \zeta_l
\]
for any $k\in\hat{I}$.
We denote by $\pi_k\colon \R^{\tilde{I}}\to\R^{\tilde{I}}$ the permutation of the $(k-1/2)$-th and $(k-1/2)$-th elements and set $\pi_s:=\pi_{k_s}\circ\cdots\circ\pi_{k_1}$.
Note that 
\[
(\pi_s(\nu))_{k+\frac{1}{2}}-(\pi_s(\nu))_{k-\frac{1}{2}}
=
N\cdot(\mu_s(\zeta))_k-\sum \zeta_l.
\]
By the definition of the sequence $(k_s)$, we have $(\mu_{s-1}(\zeta))_{k_s}>0$.
Recall that we assume $\sum\zeta_l<0$. Thus we have
\begin{equation}\label{eq->}
(\pi_{s-1}(\nu))_{k_s+\frac{1}{2}}<(\pi_{s-1}(\nu))_{k_s-\frac{1}{2}}.
\end{equation}
Let $\iota_s\colon \tilde{I}\to \tilde{I}$ be the permutation such that the sequence $((\pi_{s}(\nu))_{\iota_s(i)})_{i\in\tilde{I}}$ is increasing and set $\iota:=\iota_r$.

\begin{prop}\label{prop-lb}
Under the derived equivalence given in \S \ref{derived-equiv}, the direct summand $(\mu_\zeta(P))_k$ is a line bundle $\mathcal{L}^\zeta_k$ on $Y$. 
Moreover, it is the following map that is associated with the arrow $H^\pm_i$:
\[
\mathcal{L}^\zeta_{i\mp\frac{1}{2}}\hookrightarrow \mathcal{L}^\zeta_{i\mp\frac{1}{2}}\otimes \OO_Y(E_{\iota(i)})\simeq\mathcal{L}^\zeta_{i\pm\frac{1}{2}}.
\]
\end{prop}
\begin{proof}
We prove the claim by induction.

By the argument in the proof of Proposition \ref{smallseq}, we can verify that the sequence
\[
0\to \OO\to \OO(E^+_i)\oplus \OO(E^-_j)\to \OO(E^+_i+E^-_j)\to 0
\]
is exact if $i<j$.
Thus we have
\[
(\mu_s(P))_{k_s}\simeq {P}_{k_s}\otimes \OO(E^+_{\iota_s(k_s-1/2)}+E^-_{\iota_s(k_s+1/2)}).
\]
The second half of the claim can be verified by the explicit description of the endomorphism algebra in the proof of Proposition \ref{prop-mutation}.
\end{proof}

\begin{cor}\label{cor-vanishing}
\[
H^{1}(Y,(\mathcal{L}^\zeta_k)^{-1})=0.
\]
\end{cor}
\begin{proof}
By the argument in the proof of Proposition \ref{prop-lb}, $\mathcal{L}^\zeta_k$ is associated with the divisor which is described as a sum of divisors of the form $E^+_i+E^-_j$ ($i<j$).
We can verify the claim using the description \eqref{eq-h1} of $H^1$ of a line bundle on $Y$.
\end{proof}

\begin{prop}
The set of the torus fixed points $\mathfrak{M}_\zeta(\mathbf{v})^T$ is isolated.
\end{prop}
\begin{proof}
By Corollary \ref{cor-vanishing} we have $\OO_Y\in \mathcal{P}_\zeta$. 
Let $P_\zeta$ be the $\mu_k(A)$-module corresponding to $\OO_Y$. 
By Proposition \ref{prop-4.3} and Proposition \ref{prop-4.6}, the moduli space $\mathfrak{M}_\zeta(\mathbf{v})$ parametrizes finite dimensional quotient $\mu_k(A)$-modules $V'$ of $P_\zeta$ with $\dimv\,V'=\mu_k(\mathbf{v})$.
Note that
\[
(P_\zeta)_k=H^0(Y,(\mathcal{L}^\zeta_k)^{-1})
\]
and the $T$-weight decomposition of $H^0$ of a line bundle on $Y$ is multiplicity free.
Hence the claim follows.
\end{proof}
\end{NB}

\subsection{DT, PT and NCDT}\label{subsec-DTPTNCDT}
Note that the set ${[\mca{C}_k]}_{k\in I}$ form a basis of $H_2(Y;\Z)$. 
We identify $H_2(Y)$ with $\Z^I$.  
For $n\in\Z_{\geq 0}$ and $\beta\in H_2(Y)$, let $I_n(Y,\beta)$ denote the moduli space of ideal sheaves $\mca{I}_{Z}$ of one dimensional subschemes $\OO_Z\subset\OO_Y$ such that $\chi(\OO_Z)=n$ and $[Z]=\beta$. 
We define the {\it Donaldson-Thomas invariants} $I_{n,\beta}$ from $I_n(Y,\beta)$ using Behrend's function as is \S \ref{counting-inv} (\cite{thomas-dt}, \cite{behrend-dt}), 
and their generating function by
\[
\mca{Z}_{\mr{DT}}(Y;q,\mb{t}):=\sum_{n,\beta}I_{n,\beta}\cdot q^n\mb{t}^\beta
\] 
where $\mb{t}^\beta=\prod_{k\in ^{I}}t_k^{\beta_k}$ and $t_k$ ($k\in I$) is a formal variable.

Let $P_n(Y,\beta)$ denote the moduli space of stable pairs $(F,s)$ such that $\chi(F)=n$ and $[\mr{supp}(F)]=\beta$. 
We define the {\it Pandharipande-Thomas invariants} $P_{n,\beta}$ 
and their generating function $\mca{Z}_{\mr{PT}}(Y;q,\mb{t})$ similarly (\cite{pt1}).

We set $\zeta^\circ=(-N+1,1,\ldots,1)$, $\zeta^\pm=(-N+1\pm\varepsilon,1,\ldots,1)$ and $q=q_0\cdot q_1\cdot\cdots\cdot q_{N-1}$.    
By the result in \cite[\S 2]{nagao-nakajima}, we have the following theorem:
\begin{prop}
\[
\mca{Z}_{\mr{DT}}(Y;q,q_1^{-1},\ldots,q_{N-1}^{-1})=\mca{Z}_{\zeta^-}(\q),\quad
\mca{Z}_{\mr{PT}}(Y;q,q_1^{-1},\ldots,q_{N-1}^{-1})=\mca{Z}_{\zeta^+}(\q).
\]
\end{prop}

Let $\zeta_{\mathrm{triv}}$ be a parameter such that $(\zeta_{\mathrm{triv}})_k>0$ for any $k$.
Note that $\M{\zeta_{\mathrm{triv}}}{\vv}$ is empty unless $\vv=0$ and so $\mca{Z}_{\zeta_{\mathrm{triv}}}(\q)=1$.
Let $\zeta_{\mathrm{cyclic}}$ be a parameter such that $(\zeta_{\mathrm{cyclic}})_k<0$ for any $k$.
The invariants $D_{\zeta_{\mathrm{cyclic}}}(\vv)$ are the {\it non-commutative Donaldson-Thomas invariants} defined in \cite{szendroi-ncdt}.
We denote their generating function $\mca{Z}_{\zeta_{\mathrm{cyclic}}}(\q)$ by $\mca{Z}_{\mr{NCDT}}(\q)$.

We divide the set of positive real roots into the following two parts: 
\[
\rs^{\mr{re},+}_\pm=\{\alpha\in\rs^{\mr{re},+}\mid \pm\zeta^\circ\cdot\alpha <0\}.
\]
Applying the wall-crossing formula and comparing the equations \eqref{z} in \S \ref{counting-inv}, we obtain the following relations between generating functions:
\begin{thm}
\begin{align*}
&\mca{Z}_{\mr{NCDT}}(\q)=\left(\prod_{\alpha\in\Lambda^{\mr{re},+}_-} (1+(-1)^{\alpha_0}\mb{q}^\alpha)^{\varepsilon(\alpha)\alpha_0}\right)\cdot\mca{Z}_{\mr{DT}}(Y;q,q_1^{-1},\ldots,q_{N-1}^{-1}),\\
&\mca{Z}_{\mr{PT}}(Y;q,q_1^{-1},\ldots,q_{N-1}^{-1})=\prod_{\alpha\in\Lambda^{\mr{re},+}_+} (1+(-1)^{\alpha_0}\mb{q}^\alpha)^{\varepsilon(\alpha)\alpha_0}.
\end{align*}
\end{thm}

\begin{rem}
For the case $N_1=0$, the formula on $\mca{Z}_{\mr{NCDT}}$ and $\mca{Z}_{\mr{DT}}$ have been given in \cite{young-mckay}.
\end{rem}

We define the sets of positive real roots of the finite root system by 
\[
\rs^{\mr{fin},+}=\{\alpha_{[a,b]}:=\alpha_a+\cdots+\alpha_b \mid 0<a\leq b<N\},
\]
then we have
\[
\rs^{\mr{re},+}_+=\{\alpha+n\delta \mid \alpha\in \rs^{\mr{fin},+},\ n\geq 0\}
\]
where $\delta=(1,\ldots,1)$ is the minimal imaginary root.
Note that for $\alpha\in\rs$ we have $\varepsilon(\alpha+n\delta)=\varepsilon(\alpha)$. 
Let 
\[
M(x,q)=\prod_{n=1}^\infty(1-xq^n)^{-n}
\]
be the {\it MacMahon function}. 
\begin{cor}
\[
\mca{Z}_{\mr{PT}}(Y;q,q_1,\ldots,q_{N-1})=\prod_{0<a\leq b<N}M(q_{[a,b]},-q)^{\varepsilon(\alpha_{[a,b]})},
\]
where $q_{[a,b]}=q_a\cdot\cdots\cdot q_b$.
\end{cor}
The root lattice of the finite root system is identified with $H_2(Y)$ so that $\alpha_k$ corresponds to $[\mca{C}_k]$. 
The corollary claims the {\it Gopakumar-Vafa BPS state counts} in genus $g$ and class $\alpha$ {\it defined} in \cite[\S 3.4]{pt1} is given by 
\[
n_{g,\alpha}=\begin{cases}
-\varepsilon(\alpha_{[a,b]}) & \alpha=\alpha_{[a,b]},\, g=0,\\
0 & \text{otherwise}.
\end{cases}
\]

\begin{NB}
\section{Remarks}\label{remark}
In this section, we make some observations on how Kontsevich-Soibelman's wall-crossing formula \eqref{fp} given below (they also call the formula \eqref{fp} by "Factorization Property") would be applied in our setting.

First, we will review the work of Kontsevich-Soibelman (\cite{ks}) very briefly. 
The core of their work is the construction of the algebra homomorphism from the {\it "motivic Hall algebra"} to the {\it "quantum torus"}. 
For an $A_\infty$-category $\mca{C}$, 
the motivic Hall algebra $H(\mca{C})$ is, roughly speaking, the space of motives over the moduli $\mf{Ob}(\mca{C})$ of all objects in $\mca{C}$, with the product derived from the same diagram as the Ringel-Hall product.
The quantum torus is a deformation of a polynomial ring described explicitly in the terms of the numerical datum of $\mca{C}$. 
The homomorphism is given by taking, so to say, weighted Euler characteristics with respect to the {\it motivic weight}, where the motivic weight is defined using {\it motivic Milnor fiber} of the superpotentials coming from the $A_\infty$-structure. 
The formula \eqref{fp} is the translation of the Harder-Narashimhan property under this homomorphism. 

\begin{rem}
In the original Donaldson-Thomas invariants defined using symmetric obstruction theory, we adapt the Behrend's function as a weight (see \S \ref{counting-inv}). 
It is expected what, after taking the "quasi-classical limit" as $q\to 1$, the motivic weight would coincide with the Behrend's one. 
In \cite{ks}, the proof of the claim for some special situations and some evidences of the claim for more general situations are provided. 
\end{rem}

Now, we will explain the statement of "Factorization Property", restricting to our situation.  
We set $\rs:=\Z^{\Q_0}$ and define the skew symmetric bilinear form $\langle-,-\rangle\colon \rs\times\rs\to \Z$ by 
\[
\langle(e_i),(f_i)\rangle:=e_\infty\cdot f_0-e_0\cdot f_\infty.
\]
Let $Z\in\Hom(\rs,\C)$ be a homomorphism such that $\rm{Im}(Z(\rs^+))>0$ where $\rs^+=\Z_{\geq 0}^{\Q_0}$. 

Let $D^\mu$ be a certain ring of motives including the inverting motives $\LL^{-1}$, $[\mr{GL}(n)]^{-1}\ (n\geq 1)$ and the formal symbol $\LL^{1/2}$, where $\LL$ is the motive of the affine line. 
We have the homomorphism of rings $\phi\colon D^\mu\to \QQ(q^{1/2})$ mapping $\LL^{1/2}$ to $q^{1/2}$. 
We define the quantum torus $\mca{R}_{\rs,q}$ as the $\QQ(q^{1/2})$-algebra generated by $x_\gamma$ ($\gamma\in\rs)$ with the relation 
\[
x_\gamma x_\mu=q^{\frac{1}{2}\langle\gamma,\mu\rangle}x_{\gamma+\mu}.
\]

For a strict sector $V$ in the upper half plane, 
let $\mca{C}^{Z}_V$ denote the category of $\A$-modules which can be described as subsequent extensions of $Z$-semistable objects $E$ such that $Z(E)\in V$.
Note that $\mca{C}^{Z}_V$ does not change when $Z$ moves in $\Hom(\rs,\C)$, unless the values of $Z$ of semistable objects get close to the boundary $\partial V$.     
We define an element $A_{V,q}^Z\in\mca{R}_{V,q}$ by "weighted" counting of objects in $\mca{C}^{Z}_V$. Informally speaking,  
\[
A_{V,q}^Z:=\sum_{E\in\mr{Isom}(\mca{C}_V^Z)}\phi\left(\frac{w(E)}{[\mr{Aut}(E)]}\right)\cdot x_{\dimv(E)}\in\mca{R}_{V,q},
\]
where $w(E)\in D^\mu$ is defined by the {\it motivic Milnor fiber} of the superpotential of the $A_\infty$-algebra algebra $\Ext^*(E,E)$.

\begin{figure}[htbp]
  \centering
  \input{pic6.tpc}
  \caption{$V$}
  \label{fig:V}
\end{figure}
Assume that $V$ is decomposed into a disjoint union $V=V_1\sqcup V_2$ in the clockwise order. 
Then the "Factorization Property" in \cite{ks} claims that 
\begin{equation}\label{fp}
A^Z_{V,q}=A^Z_{V_1,q}\cdot A^Z_{V_2,q}.
\end{equation}
The key fact is, as we mentioned above, the existence of the algebra homomorphism from the motivic Hall algebra to $\mca{R}_{V,q}$. 
Although the category of perverse coherent systems is not Calabi-Yau, Proposition \ref{prop3.1.} would assure the existence of the algebra homomorphism in our case.
We can define an element $A^Z_{V,\mr{mot}}$ in the motivic Hall algebra and the equation $A^Z_{V,\mr{mot}}=A^Z_{V_1,\mr{mot}}\cdot A^Z_{V_2,\mr{mot}}$ follows from the Harder-Narashimhan property. 
Now the element $A^Z_{V,q}$ is the image of $A^Z_{V,\mr{mot}}$ under the algebra homomorphism.

Now we end up with reviewing and begin to explain how to apply \eqref{fp} in our setting.
Since we are interested in $\A$-modules $V$ with $V_\infty\simeq \C$, we will work on the quotient algebra
\[
\mca{R}'_{\rs,q}:=\mca{R}_{\rs,q}/(x_{\mb{e}})_{\{\mb{e}\,\mid\,e_\infty\geq 2\}}.
\]
Consider the wall in $\Hom(\rs,\C)$ such that $\mb{e}\in\rs^+$ with $e_\infty=0$ and $\mb{f}\in\rs^+$ with $f_\infty=1$ are send on a same half line in the upper half plane. 
Assume $\mb{e}\in\rs$ is primitive (i.e. $\{e_i\}_{i\in Q_0}$ are coprime to each other) and $\mb{f}-\mb{e}\notin \rs^+$. 
Let $\rs_0\in\rs$ be the sublattice generated by $\mb{e}$ and $\mb{f}$, $l_k$ be the half line passing through $k\cdot\mb{e}+\mb{f}$ and $l_\infty$ be the half line passing through $\mb{e}$.
\begin{figure}[htbp]
  \centering
  \input{pic4.tpc}
  \caption{$\Lambda_0$}
\end{figure}
Take $Z^+$ and $Z^-$ from the opposite side of the wall so that $l_1,l_2,\ldots,l_\infty$ are mapped on the upper half plane in the clockwise (resp. anticlockwise) way by $Z^+$ (resp. $Z^-$).   
The "Factorization Property" claims 
\[
A^{Z^+}_{l_1}\cdot A^{Z^+}_{l_2}\cdot\cdots\cdot A^{Z^+}_{l_\infty}=A^{Z^-}_{l_\infty}\cdot\cdots\cdot A^{Z^-}_{l_2}\cdot A^{Z^-}_{l_1}
\]
in $\mca{R}'_{\rs,q}$. We denote 
\[
\prod^{\rightarrow}_{k}A^{Z^+}_{l_k}:=A^{Z^+}_{l_1}\cdot A^{Z^+}_{l_2}\cdot\cdots,\quad \prod^{\leftarrow}_{k}A^{Z^-}_{l_k}:=\cdots\cdot A^{Z^-}_{l_2}\cdot A^{Z^-}_{l_1}.
\]
Note that $A^{Z^+}_{l_\infty}=A^{Z^-}_{l_\infty}$ and we denote this by $A^{Z}_{l_\infty}$.
Then the above equation is described as following:
\[
\prod^{\rightarrow}_{k}A^{Z^+}_{l_k}=A^{Z}_{l_\infty}\cdot\left(\prod^{\rightarrow}_{k}A^{Z^-}_{l_k}\right)\cdot{A^{Z}_{l_\infty}}^{-1}.
\]
An element $A$ of $\mca{R}'_{\rs,q}$ can be uniquely described in the following form: 
\[
A=\sum_{\mb{e};\,e_\infty=0}(a_{\mb{e}}(q)\cdot x_{\mb{e}})+x_\infty\cdot \sum_{\mb{e};\,e_\infty=0}(b_{\mb{e}}(q)\cdot x_{\mb{e}}).
\]
We denote $\sum_{\mb{e};\,e_\infty=0}(b_{\mb{e}}(q)\cdot x_{\mb{e}})$ by $A^{x_\infty}$.
Then
\[
(q-1)\cdot\left(\prod^{\rightarrow}_{k}A^{Z^+}_{l_k}\right)^{x_\infty}\Bigg|_{q=1}
\]
makes sense and 
would coincide with the generating function of virtual counting of the moduli spaces we study in \S \ref{counting}. 
Note that we have 
\[
\langle k\cdot\mb{e}+\mb{f},\mb{e}\rangle=e_0.
\]
and so
\[
x_{m\cdot \mb{e}}\cdot x_{k\cdot\mb{e}+\mb{f}}=q^{m\cdot e_0}x_{k\cdot\mb{e}+\mb{f}}\cdot x_{m\cdot \mb{e}}.
\]
Identify $A^{Z^+}_{l_\infty}\in \mca{R}'_{\rs,q}$ with the polynomial in $x_{\mb{e}}$, then we have
\begin{align*}
\prod^{\rightarrow}_{k}A^{Z^+}_{l_k}&=
A^{Z}_{l_\infty}(x_{\mb{e}})\cdot
\left(\prod^{\leftarrow}_{k}A^{Z^-}_{l_k}\right)\cdot
{A^{Z}_{l_\infty}}(x_{\mb{e}})^{-1}\\
&=\left(\prod^{\leftarrow}_{k}A^{Z^-}_{l_k}\right)\cdot A^{Z}_{l_\infty}(q^{e_0}x_{\mb{e}})\cdot{A^{Z}_{l_\infty}}(x_{\mb{e}})^{-1}
\end{align*} 
in $\mca{R}'_{\rs,q}$. 
Now, if we can compute $A^{Z}_{l_\infty}(q^{e_0}x_{\mb{e}})\cdot{A^{Z}_{l_\infty}}(x_{\mb{e}})^{-1}\big|_{q=1}$, we get wall-crossing formulas.

Here again, observations in \cite{ks} will help us.
We put $t:=x_{\mb{e}}$.
Assume that we have the unique simple object $E$ on $l_\infty$ and $\dimv\,E=\mb{e}$. 
Let $B_E$ be the algebra generated by $\Ext^1(E,E)$ with relations defined from the superpotential $W_E$.
Then we have 
\begin{equation}\label{ks-obs}
A^{Z}_{l_\infty}(qt)=A^{Z}_{l_\infty}(t)\cdot f(t)
\end{equation}
where $f(t)$ is obtained by counting pairs of cyclic $B_E$-modules and their cyclic vectors. 
Applying this formula repeatedly we have
\[
A^{Z}_{l_\infty}(q^{e_0}t)\cdot A^{Z}_{l_\infty}(t)^{-1}\big|_{q=1}=\left(f(t)|_{q=1}\right)^{e_0}.
\]
\begin{ex}
\begin{enumerate}
\item[(1)] Assume $\ext^1(E,E)=0$. The algebra $B_E$ is trivial. Hence we have $f(t)|_{q=1}=1+t$. This corresponds to the formula in Theorem \ref{thm3.13.}. 
\item[(2)] Assume $\ext^1(E,E)=1$ and $B_E\simeq \C[z]$. 
In this case we have $f(t)|_{q=1}=(1-t)^{-1}$. 
This corresponds to the formula in Theorem \ref{thm3.16.}.  
\item[(3)] 
Let us consider the wall corresponding to the imaginary root. 
The set of simple objects on this wall is $\{\OO_y\mid y\in Y\}$.
By the same argument as they show the above equation (\ref{ks-obs}) in \cite{ks}, we would have 
\[
A^{Z}_{l_\infty}(qt)=A^{Z}_{l_\infty}(t)\cdot f(t)
\] 
where $f(t)$ is obtained by counting $0$-dimensional closed subscheme of $Y$. 
By the results of \cite{mnop} and \cite{behrend-fantechi} we have
\[
f(t)|_{q=1}=M(-t)^{e(Y)} 
\]
where 
\[
M(t):=\prod_{n=1}^{\infty}(1-t^n)^{-n}
\]
is the MacMahon function. This provides DT-PT correspondence in our situation. 
\end{enumerate}
\end{ex}
\end{NB}

\bibliographystyle{amsalpha}
\bibliography{bib-ver6}

{\tt
\noindent Kentaro Nagao

\noindent Research Institute for Mathematical Sciences, Kyoto University, Kyoto 606-8502, Japan

\noindent kentaron@kurims.kyoto-u.ac.jp
}

\end{document}